\documentclass{amsart}
\usepackage{graphicx}
\usepackage{amsmath}%in Winedt
\usepackage{amssymb}
\usepackage{amsthm}
\newtheorem{thm}{Theorem}[section]
\newtheorem{lem}{Lemma}[section]

\newtheorem{conj}{Conjecture}[section]

\title{On the Crossing Number of Some Complete Multipartite Graphs}
\author{Pak Tung Ho}
\begin{document}

\begin{abstract}
In this paper, we find the crossing number of the complete
multipartite graphs $K_{1,1,1,1,n}$, $K_{1,2,2,n}$, $K_{1,1,1,2,n}$
and $K_{1,4,n}$. Our proof depends on Kleitman's results for the
complete bipartite graphs [D. J. Kleitman, {\it The crossing number
of $K_{5,n}$,} J. Combinatorial Theory 9 (1970) 315-323].
\end{abstract}

\maketitle

\section{Introduction}
In this paper, we consider $G$ as a 1-complex, that is the union of
vertices and edges and, if two edges $e$ and $e'$ intersect, $e$
intersects with $e'$ in one of the endpoints. An immersion $\phi$ of
$G$ into the 2-dimensional Euclidean space $\mathbb{R}^2$ is said to
be \textit{good}, if the following conditions are satisfied:
\begin{itemize}
    \item[(i)] $\phi |\phi^{-1}(\phi(V))$ and $\phi |e$ are one
    to one, where $e$ is an edge.
    \item[(ii)] For any point $p$ in $\mathbb{R}^2$,
    $\phi^{-1}(p)$ consists of at most two points.
    \item[(iii)] $\phi(e_i)\cap\phi(e_j)$ consists of at most one
    point for distinct edges $e_i$ and $e_j$.
\end{itemize}
Note that
$\phi(\overset{\circ}e_i)\cap\phi(\overset{\circ}e_j)=\varnothing$
for adjacent edges $e_i$ and $e_j$, where $\overset{\circ}e$ is the
interior of $e\in E$. We will call the image $\phi(G)$ of a good
immersion $\phi$ a \textit{drawing} of $G$.

Let $A$ and $B$ be subsets of $E$. Then the cardinality of
$\{\phi(\overset{\circ}a)\cap\phi(\overset{\circ}b)|\mbox{ }a\in
A,\mbox{ }b\in B\}$ is denoted by $cr_\phi(A,B)$. Especially,
$cr_\phi(A,A)$ will be denoted by $cr_\phi(A)$. We call
$cr_\phi(E)$ the crossing of $\phi$. The \textit{crossing number}
$cr(G)$ of a graph $G$ is the minimum crossing number among all
good immersions.

Let $A$ be a nonempty subset of $V$ or of $E$, for a graph $G$. Then
$\langle A\rangle$ denotes the subgraph of $G$ induced by $A$. The
set of edges which are incident with a vertex $v$ is denoted by
$E(v)$. For a complete $k$-partite graph $K_{a_1,a_2,...,a_k}$ with
the partition $(A_1,A_2,....,A_k)$ and the edge set $E$, where
$|A_i|=a_i$, we will write $E_{A_iA_j}$ for the edge sets of
$\langle A_i\cup A_j\rangle$.

We note the following formulas, which can be shown easily.
\begin{eqnarray}
\label{3} cr_\phi(A\cup B)&=&cr_\phi(A)+cr_\phi(B)+cr_\phi(A,B)\\
\label{4} cr_\phi(A,B\cup C)&=&cr_\phi(A,B)+cr_\phi(A,C)
\end{eqnarray}
where $A$, $B$ and $C$ are mutually disjoint subsets of $E$.

A good, updated survey on crossing numbers is \cite{Szekely}. A
longstanding problem in the theory of crossing number is
Zarankiewicz's conjecture, which asserts that the crossing number of
the complete graphs $K_{m,n}$ is given by
\begin{eqnarray}\label{21}
cr(K_{m,n})=\lfloor\frac{m}{2}\rfloor\lfloor\frac{m-1}{2}
\rfloor\lfloor\frac{n}{2}\rfloor\lfloor\frac{n-1}{2}\rfloor.
\end{eqnarray}
It is only known to be true for $m\leq 6$ \cite{Kleitman}; and for
$m=7$ and $n\leq 10$ \cite{Woodall}. Recently, in \cite{Richter}, E.
deKlerk et al. give a new lower bound for the crossing number of
$K_{m,n}$ and $K_n$. In the following, $Z(m,n)$ will denote the
right member of (\ref{21}).

It is natural to ask generalize the Zarankiewicz's conjecture and
ask: What is the crossing number for the complete multipartite
graph? In this paper, using Kleitman's result, we will determine the
crossing number of the multipartite graphs $K_{1,1,1,1,n}$,
$K_{1,2,2,n}$, $K_{1,1,1,2,n}$ and $K_{1,4,n}$ as follows:
$cr(K_{1,1,1,1,n})=Z(4,n)+n$; $
cr(K_{1,2,2,n})=Z(5,n)+\displaystyle\lfloor\frac{3n}{2}\rfloor;$
$cr(K_{1,1,1,2,n})=Z(5,n)+2n;$
$cr(K_{1,4,n})=Z(5,n)+2\displaystyle\lfloor\frac{n}{2}\rfloor$.

The arguments used to obtain the crossing numbers of the graphs
except $K_{1,4,n}$ are borrowed heavily from Asano \cite{Asano}, who
determines the crossing number of $K_{1,3,n}$ and $K_{2,3,n}$ as
follows:$cr(K_{1,3,n})=Z(4,n)+\displaystyle\lfloor\frac{n}{2}\rfloor$;
$ cr(K_{2,3,n})=Z(5,n)+n.$ In addition to Asano's paper, our proof
is based on the results of Harborth in \cite{Harborth}, who have
presented all the non-isomorphic drawings of $K_{2,2}$, $K_{2,3}$
and $K_{3,3}$. To obtain the crossing numbers of $K_{1,4,n}$, we use
the basic counting argument. On the other hand, the crossing number
of $K_{1,3,n}$ can be determined by this counting argument.

\section{Preliminary Results}
In this section, we provide some results which can be used to obtain
the required crossing numbers. Lemma \ref{24} and \ref{35} will be
used to obtain the crossing number of $K_{1,2,2,n}$.

\begin{lem}\label{24}
For any good immersion $\phi$ of $K_{1,2,2}$, there is at most 1
region whose boundary contains at least 4 vertices in $V$.
\end{lem}
\begin{proof}
From \cite{Harborth}, we know that there are 6 non-isomorphic
drawings of $K_{2,3}$, namely, the drawings in Figure 1.
\begin{center}
\begin{picture}(0,0)(170,35)
\put(0,20){$\circ$} \put(40,20){$\bullet$} \put(0,-20){$\bullet$}
\put(40,-20){$\circ$} \put(20,0){$\circ$}

\put(3,21){\line(0,-1){39}} \put(4,23){\line(1,0){39}}
\put(43,21){\line(0,-1){37}} \put(4,-17){\line(1,0){37}}
\put(3,-17){\line(1,1){19}} \put(42,21){\line(-1,-1){18}}

%%%%%%%%%%%%%%%%%%%

\put(55,20){$\circ$} \put(95,20){$\circ$} \put(55,-20){$\bullet$}
\put(95,-20){$\bullet$} \put(75,-30){$\circ$}

\put(58,21){\line(0,-1){39}} \put(98,21){\line(0,-1){37}}
\put(58,-17){\line(1,1){38}} \put(98,-17){\line(-1,1){39}}
\put(79,-27){\line(2,1){18}} \put(76,-27){\line(-2,1){18}}

%%%%%%%%%%%%%%%%%%%%%%%%%%%%%%%

\put(110,20){$\circ$} \put(150,20){$\circ$} \put(110,-20){$\bullet$}
\put(150,-20){$\bullet$} \put(125,-10){$\circ$}

\put(113,21){\line(0,-1){39}} \put(153,21){\line(0,-1){37}}
\put(153,-17){\line(-1,1){39}} \put(113,-17){\line(1,0){20}}
\put(133,-17){\line(1,2){19}} \put(113,-17){\line(3,2){13}}
\put(126,-8){\line(-1,0){20}} \put(106,-8){\line(0,-1){18}}
\put(106,-26){\line(1,0){47}}  \put(153,-26){\line(0,1){10}}

%%%%%%%%%%%%%%%%%%%%%%%%%%%%%%%%

\put(165,20){$\circ$} \put(205,20){$\circ$} \put(165,-20){$\bullet$}
\put(205,-20){$\bullet$} \put(180,-10){$\circ$}

\put(168,21){\line(0,-1){39}} \put(208,21){\line(0,-1){37}}
\put(208,-17){\line(-1,1){39}} \put(168,-17){\line(1,0){20}}
\put(188,-17){\line(1,2){19}} \put(168,-17){\line(3,2){13}}
\put(208,-19){\line(-2,1){23}}

%%%%%%%%%%%%%%%%%%%%%%%%%%%%%%%%

\put(218,20){$\circ$} \put(260,20){$\bullet$} \put(218,-20){$\circ$}
\put(260,-20){$\bullet$}
 \put(218,0){$\circ$}

\put(222,22){\line(1,0){40}} \put(222,22){\line(1,-1){40}}
\put(262,22){\line(-1,-1){40}} \put(262,-18){\line(-1,0){40}}
\put(222,2){\line(2,1){40}} \put(222,2){\line(2,-1){40}}

%%%%%%%%%%%%%%%%%%%%%%%%%

\put(290,20){$\circ$} \put(330,20){$\circ$} \put(290,-20){$\bullet$}
\put(330,-20){$\bullet$} \put(305,-10){$\circ$}

\put(293,21){\line(0,-1){39}} \put(333,21){\line(0,-1){37}}
\put(333,-17){\line(-1,1){39}} \put(293,-17){\line(1,0){20}}
\put(313,-17){\line(1,2){19}} \put(333,-19){\line(-2,1){24}}
\put(307,-6){\line(0,1){40}} \put(307,34){\line(-1,0){25}}
\put(282,34){\line(0,-1){51}} \put(282,-17){\line(1,0){10}}

%%%%%%%%%%%%%%%%
\put(20,-40){$D_1$} \put(72,-40){$D_2$} \put(125,-40){$D_3$}
\put(182,-40){$D_4$} \put(239,-40){$D_5$} \put(300,-40){$D_6$}

\put(150,-60){Figure 1.}
\end{picture}
\end{center}
$$$$\\\\\\\\\\

To obtain a drawing of $K_{1,2,2}$ from these drawings of $K_{2,3}$,
we have to choose a vertex from the partition of $K_{2,3}$
containing 3 vertices and draw edges connecting this vertex and the
other two vertices in the partition. For $2\leq i\leq 6$, $D_i$ has
no more than 1 region whose boundary contains at least 4 vertices.
Therefore, $D_1$ is the only drawing of $K_{2,3}$ which can be used
to obtain a possible counterexample of Lemma \ref{24}.

\begin{center}
\begin{picture}(0,0)(70,20)
\put(0,20){$\circ$} \put(40,20){$\bullet$} \put(0,-20){$\bullet$}
\put(40,-20){$\circ$} \put(19,0){$\blacktriangle$}

\put(3,21){\line(0,-1){39}} \put(4,23){\line(1,0){39}}
\put(43,21){\line(0,-1){37}} \put(4,-17){\line(1,0){37}}
\put(3,-17){\line(1,1){19}} \put(42,21){\line(-1,-1){18}}

\put(24,1){\line(1,-1){17}} \put(22,4){\line(-1,1){18}}
\end{picture}
\begin{picture}(0,0)(-30,20)
\put(0,20){$\circ$} \put(40,20){$\bullet$} \put(0,-20){$\bullet$}
\put(40,-20){$\circ$} \put(20,0){$\blacktriangle$}

\put(3,21){\line(0,-1){39}} \put(4,23){\line(1,0){39}}
\put(43,21){\line(0,-1){37}} \put(4,-17){\line(1,0){37}}
\put(3,-17){\line(1,1){19}} \put(42,21){\line(-1,-1){18}}

\put(24,1){\line(1,-1){17}} \put(23,2){\line(0,-1){30}}
\put(23,-28){\line(-1,0){30}} \put(-7,-28){\line(0,1){51}}
\put(-7,23){\line(1,0){8}}

\put(-55,-50){Figure 2.}
\end{picture}
\end{center}
$$$$\\\\\\

In $D_1$, we have to choose a vertex from the partition of $K_{2,3}$
with 3 vertices (that is, the white vertices in $D_1$) and draw
edges connecting this vertex and the other two vertices in the
partition. Up to isomorphism, the only possible drawing of
$K_{1,2,2}$ is as in Figure 2, which proves the lemma.
\end{proof}

\begin{lem}\label{35}
There are 3 non-isomorphic drawings of $K_{1,2,2}$ such that it
contains a region whose boundary contains more than 4 vertices in
$V$.
\end{lem}
\begin{proof}
From \cite{Harborth} again, we know that there are 6 non-isomorphic
drawings of $K_{2,3}$, namely, the drawings in Figure 1. To obtain a
drawing of $K_{1,2,2}$ from these drawing of $K_{2,3}$ such that
there is a region whose boundary contains more than 4 vertices, the
only candidates are $D_2$ and $D_5$. To preserve the region whose
boundary containing more than $4$ vertices, the only possible
drawings of $K_{1,2,2}$, up to isomorphism, are shown in Figure 3.\\

\begin{center}
\begin{picture}(0,0)(105,15)

\put(5,20){$\circ$} \put(45,20){$\circ$} \put(5,-20){$\bullet$}
\put(45,-20){$\bullet$} \put(23,-30){$\blacktriangle$}

\put(8,21){\line(0,-1){39}} \put(48,21){\line(0,-1){37}}
\put(8,-17){\line(1,1){38}} \put(48,-17){\line(-1,1){39}}
\put(29,-27){\line(2,1){18}} \put(26,-27){\line(-2,1){18}}
\put(28,-27){\line(2,5){19}} \put(28,-27){\line(-2,5){19}}
\end{picture}
\begin{picture}(0,0)(30,15)

\put(5,20){$\circ$} \put(44,20){$\blacktriangle$}
\put(5,-20){$\bullet$} \put(45,-20){$\bullet$}
\put(25,-30){$\circ$}

\put(8,21){\line(0,-1){39}} \put(48,21){\line(0,-1){37}}
\put(8,-17){\line(1,1){38}} \put(48,-17){\line(-1,1){39}}
\put(29,-27){\line(2,1){18}} \put(26,-27){\line(-2,1){18}}
\put(28,-27){\line(2,5){19}} \put(48,23){\line(-1,0){39}}
\end{picture}
\begin{picture}(10,0)(0,15)

\put(55,20){$\bullet$} \put(95,20){$\circ$}
\put(55,-20){$\bullet$} \put(95,-20){$\circ$}
\put(74,-30){$\blacktriangle$}

\put(96,-17){\line(-1,0){39}} \put(96,23){\line(-1,0){37}}
\put(58,-17){\line(1,1){38}} \put(98,-17){\line(-1,1){39}}
\put(79,-27){\line(2,1){18}} \put(76,-27){\line(-2,1){18}}
\put(78,-27){\line(2,5){19}} \put(78,-27){\line(-2,5){20}}

%%%%%%%%%%%%%%%%
\put(-90,-43){$D_1$} \put(-10,-43){$D_2$} \put(75,-43){$D_3$}

\put(-25,-65){Figure 3.}
\end{picture}
\end{center}
$$$$\\\\\\\\
\end{proof}

\section{Crossing number of $K_{1,1,1,1,n}$}

In this section, we will prove
\begin{thm}\label{1}
The crossing number of the complete 5-partite graph
$K_{1,1,1,1,n}$ is given by  $$cr(K_{1,1,1,1,n})=Z(4,n)+n.$$
\end{thm}
\begin{proof}
Let $(X,Y,S,T,Z)$ be partition of $K_{1,1,1,1,n}$ such that $X=\{
x_1\}$, $Y=\{ y_1\}$, $S=\{ s_1\}$, $T=\{ t_1\}$ and
$Z=\displaystyle\bigcup_{i=1}^n\{z_i\}$. To show that
$cr(K_{1,1,1,1,n})\leq Z(4,n)+n,$ see Figure 4 for $n=4$ and it can
be generalized to $n$.

\begin{center}
\begin{picture}(60,30)(-30,20)
\put(0,0){$\bullet$} \put(1,39){$\bullet$}
\put(-11,-40){$\bullet$} \put(0,-81){$\bullet$}

\put(5,3){$x_1$} \put(10,41){$s_1$} \put(-17,-45){$y_1$}
\put(-13,-81){$t_1$}

\put(2,2){\line(0,-1){80}} \put(2,2){\line(-1,-4){10}}
\put(-8,-38){\line(1,-4){10}} \put(3,2){\line(0,1){40}}
\put(-10,-38){\line(1,6){13}}

\put(20,-20){$\bullet$} \put(60,-22){$\bullet$}
\put(-40,-20){$\bullet$} \put(-80,-21){$\bullet$}

\put(8,-18){$z_2$} \put(-49,-20){$z_1$} \put(67,-20){$z_4$}
\put(-89,-21){$z_3$}

\put(22,-18){\line(-1,1){20}} \put(24,-18){\line(-1,3){20}}
\put(24,-17){\line(-3,-2){32}} \put(24,-18){\line(-1,-3){20}}

\put(62,-18){\line(-3,1){60}} \put(64,-18){\line(-1,1){60}}
 \put(64,-20){\line(-4,-1){71}} \put(64,-18){\line(-1,-1){60}}

\put(-38,-18){\line(2,1){40}} \put(-38,-18){\line(2,3){40}}
\put(-38,-18){\line(3,-2){28}} \put(-38,-18){\line(2,-3){40}}

\put(-78,-18){\line(4,1){80}} \put(-78,-17){\line(4,3){80}}
\put(-78,-20){\line(4,-1){70}} \put(-78,-20){\line(4,-3){80}}

\put(95,-18){\line(-3,2){90}} \put(95,-18){\line(-3,-2){90}}

\put(-20,-108){Figure 4.}
\end{picture}
\end{center}
$$$$\\\\\\\\\\\\\\\\

Therefore it suffices to show that
\begin{eqnarray}\label{2}
cr(K_{1,1,1,1,n})\geq Z(4,n)+n,
\end{eqnarray}
Note that $K_{1,1,1,1,1}$ is isomorphic to the complete graph of
order 5, $K_5$. Therefore, $cr(K_{1,1,1,1,1})=cr(K_5)=1$, which
shows that (\ref{2}) is true for $n=1$.

 Now suppose $n\geq 1$. If
$cr(K_{1,1,1,1,n})<Z(4,n)+n$, then there exists a good immersion
$\phi$ of $K_{1,1,1,1,n}$ such that
\begin{eqnarray}\label{6}
cr_{\phi}(E)<Z(4,n)+n.
\end{eqnarray}

Let $W=E_{XY}\cup E_{XS}\cup E_{XT} \cup E_{YS}\cup E_{YT}\cup
E_{ST}$. Then, by (\ref{3}) and (\ref{4}), we have
\begin{eqnarray}\label{5}
\begin{array}{rcl}
cr_\phi(E)&=&\displaystyle cr_\phi({W})+cr_\phi(\bigcup_{i=1}^n
E(z_i))+ cr_\phi({W},\bigcup_{i=1}^n E(z_i))\\
&=&\displaystyle cr_\phi({W})+cr_\phi(\bigcup_{i=1}^n
E(z_i))+\sum_{i=1}^n cr_\phi({W},E(z_i)).
\end{array}
\end{eqnarray}
Since $\langle\displaystyle\bigcup_{i=1}^n E(z_i)\rangle\cong
K_{4,n}$, by (\ref{21}), we have
\begin{eqnarray}\label{7}
cr_\phi(\bigcup_{i=1}^n E(z_i))\geq Z(4,n).
\end{eqnarray}
Therefore, (\ref{6}), (\ref{5}) and (\ref{7}) gives
$cr_\phi({W},E(z_i))=0$ for some $i.$ By reordering of $i$, we may
suppose
\begin{eqnarray}\label{22}
cr_\phi({W},E(z_1))=0.
\end{eqnarray}
Then $\phi(\langle W\rangle)$ divides $\mathbb{R}^2$ into regions
and the condition (\ref{22}) implies that $\phi(X\cup Y\cup S\cup
T)$ is contained in the boundary of one of the regions. Denote
$F=W\cup E(z_1)$. Then, without loss of generality, we may assume
that $\phi| \langle F\rangle$ has the drawing as in Figure 5.

\begin{center}
\begin{picture}(0,0)(20,20)
\put(-20,0){$\bullet$} \put(20,0){$\bullet$}
\put(-20,-40){$\bullet$} \put(20,-40){$\bullet$}
\put(60,-20){$\circ$} \put(68,-20){$z_1$}

\put(-18,2){\line(0,-1){40}} \put(-18,2){\line(1,0){40}}
\put(-18,2){\line(1,-1){40}} \put(22,2){\line(0,-1){40}}
\put(22,2){\line(-1,-1){40}} \put(22,-38){\line(-1,0){40}}

\put(61,-17){\line(-2,1){37}} \put(61,-17){\line(-2,-1){37}}
\put(61,-17){\line(-1,-1){40}} \put(21,-57){\line(-1,0){38}}
\put(-17,-57){\line(0,1){20}} \put(61,-17){\line(-1,1){40}}
\put(21,23){\line(-1,0){38}} \put(-17,23){\line(0,-1){20}}

\put(0,-80){Figure 5.}
\end{picture}
\end{center}
$$$$\\\\\\\\\\\\

For $i\geq 2$, no matter which regions $\phi(z_i)$ is drawn,
\begin{eqnarray}
\label{8} cr_\phi(F,E(z_i))\geq 2.
\end{eqnarray}
Also, by (\ref{3}) and (\ref{4}), we have
\begin{eqnarray}\label{10}
cr_\phi(E)=cr_\phi(F)+cr_\phi(\bigcup_{i=2}^n E(z_i))+\sum_{i=2}^n
cr_\phi(F,E(z_i)).
\end{eqnarray}

Since $\langle\displaystyle\bigcup_{i=2}^n E(z_i)\rangle\cong
K_{4,n-1}$, by (\ref{21}), we have
\begin{eqnarray}\label{9} cr_\phi(\bigcup_{i=2}^n E(z_i))\geq
Z(4,n-1).
\end{eqnarray}
Combining (\ref{8}), (\ref{10}) and (\ref{9}) and the fact that
$cr_\phi(F)=1$, we have $cr_\phi(E)\geq 1+Z(4,n-1)+2(n-1)\geq
Z(4,n)+n,$ which contradicts (\ref{6}).
\end{proof}

\section{Crossing number of $K_{1,2,2,n}$}

This section is devoted to proving the following theorem.

\begin{thm}
The crossing number of the complete 4-partite graph $K_{1,2,2,n}$
is given by
$$cr(K_{1,2,2,n})=Z(5,n)+\lfloor\frac{3n}{2}\rfloor.$$
\end{thm}
\begin{proof}
Let $(X,Y,U,Z)$ be partition of $K_{1,2,2,n}$ such that $X=\{x_1\}$,
$Y=\{ y_1, y_2\}$, $U=\{ u_1, u_2\}$ and
$Z=\displaystyle\bigcup_{i=1}^n\{z_i\}$. To show that
$cr(K_{1,2,2,n})\leq
Z(5,n)+\displaystyle\lfloor\frac{3n}{2}\rfloor,$ consider Figure 6
for $n=4$, and it can be generalized to general $n$.

\begin{center}
\begin{picture}(60,30)(-30,30)
\put(0,0){$\bullet$} \put(1,40){$\bullet$}
\put(-11,-40){$\bullet$} \put(0,-81){$\bullet$}
\put(0,-121){$\bullet$}

\put(5,3){$y_1$} \put(-13,41){$u_2$} \put(-5,-45){$u_1$}
\put(0,-66){$y_2$} \put(-15,-121){$x_1$}

\put(2,2){\line(-1,-4){10}} \put(-8,-38){\line(1,-4){10}}
\put(3,2){\line(0,1){40}}

\put(20,-20){$\bullet$} \put(60,-22){$\bullet$}
\put(-40,-20){$\bullet$} \put(-80,-21){$\bullet$}

\put(8,-18){$z_3$} \put(-49,-20){$z_1$} \put(67,-20){$z_4$}
\put(-89,-21){$z_2$}

\put(22,-18){\line(-1,1){20}} \put(24,-18){\line(-1,3){20}}
\put(24,-17){\line(-3,-2){32}} \put(24,-18){\line(-1,-3){20}}

\put(62,-18){\line(-3,1){60}} \put(64,-18){\line(-1,1){60}}
 \put(64,-20){\line(-4,-1){71}} \put(64,-18){\line(-1,-1){60}}

\put(-38,-18){\line(2,1){40}} \put(-38,-18){\line(2,3){40}}
\put(-38,-18){\line(3,-2){28}} \put(-38,-18){\line(2,-3){40}}

\put(-78,-18){\line(4,1){80}} \put(-78,-17){\line(4,3){80}}
\put(-78,-20){\line(4,-1){70}} \put(-78,-20){\line(4,-3){80}}

\put(95,-18){\line(-3,2){90}} \put(95,-18){\line(-3,-2){90}}

\put(4,-119){\line(1,5){20}} \put(3,-119){\line(0,1){40}}
\put(4,-119){\line(3,5){60}} \put(2,-119){\line(-2,5){40}}
\put(2,-119){\line(-2,5){40}} \put(1,-119){\line(-4,5){80}}

\put(-122,-17){\line(6,1){123}} \put(-122,-17){\line(6,-5){123}}

 \put(-10,-38){\line(1,-6){13}}

 \put(4,-119){\line(3,2){153}} \put(157,-17){\line(-5,2){153}}

 \put(-20,-149){Figure 6.}
\end{picture}
\end{center}
$$$$\\\\\\\\\\\\\\\\\\\\\\\\\\

Therefore it is
sufficient to prove that
\begin{eqnarray}\label{26}
cr(K_{1,2,2,n})\geq Z(5,n)+\displaystyle\lfloor\frac{3n}{2}\rfloor.
\end{eqnarray}
We will prove (\ref{26}) by induction on $n$. For $n=1$,
$K_{1,2,2,1}$ contains $K_{3,3}$ and it is clear that $K_{3,3}$ is
nonplanar, therefore $cr(K_{1,2,2,1})\geq 1$. Therefore (\ref{26})
is true for $n=1$. For $n=2$, consider a good immersion $\phi$ of
$K_{1,2,2,2}$. Note that $\langle E_{XY}\cup E_{XU}\cup
E_{YU}\rangle\cong K_{1,2,2}$, by Lemma \ref{24}, any drawings of
$K_{1,2,2}$ has at most one region which contains at least 4
vertices. We will consider
three cases:\\
\textbf{Case A.} If all the regions of the drawing $\phi(\langle
E_{XY}\cup E_{XU}\cup E_{YU}\rangle)$ whose boundary contains less
than 4 vertices in $\phi(X\cup Y\cup U)$, then $cr_\phi(E(z_i),
E_{XY}\cup E_{XU}\cup E_{YU})\geq 2$ for $i=1, 2$ which implies
that
$cr_\phi(E)\geq 4$. Therefore (\ref{26}) is true for this case.\\
\textbf{Case B.} Suppose that the drawing $\phi(\langle E_{XY}\cup
E_{XU}\cup E_{YU}\rangle)$ contains a unique region, $f$, whose
boundary contains exactly 4 vertices in $\phi(X\cup Y\cup U)$. If
there exists $z_i$, say $z_1$, such that $\phi(z_i)$ is not
contained in $f$, then $cr_\phi(E(z_1), E_{XY}\cup E_{XU}\cup
E_{YU})\geq 2$. Note also that $cr_\phi(E(z_2), E_{XY}\cup
E_{XU}\cup E_{YU})\geq 1$ since all drawings of $K_{1,2,2}$ have
only one region which contains at least 4 vertices in $\phi(V)$.
This gives $cr_\phi(E)\geq 3$. Now suppose that all $z_i$ whose
image under $\phi$ are contained in $f$. If $cr_\phi(E(z_i),
E_{XY}\cup E_{XU}\cup E_{YU})>1$ for some $i$, then the proof is the
same as before.  If $cr_\phi(E(z_i), E_{XY}\cup E_{XU}\cup
E_{YU})=1$ for $i=1, 2$, then both $\phi(z_1)$ and $\phi(z_2)$ must
lies in $f$. Thus $cr_\phi(E(z_1),E(z_2))\geq 1$ (see Figure 7).
This gives $cr_\phi(E)\geq 3$. Therefore (\ref{26}) is true for this
case.

\begin{center}
\begin{picture}(0,0)(80,10)
\put(20,0){$\bullet$}
 \put(20,-40){$\bullet$}
 \put(-20,-40){$\bullet$}
 \put(-20,0){$\bullet$}

\put(22,2){\line(0,-1){40}} \put(22,2){\line(-1,0){40}}
\put(22,-38){\line(-1,0){40}} \put(-18,-38){\line(0,1){40}}

\put(22,2){\line(-1,-1){20}} \put(-18,2){\line(1,-1){40}}
\put(-18,-38){\line(1,1){20}}

\put(2,-18){\line(1,0){40}} \put(0,-20){$\bullet$}
\put(8,-16){$z_1$}

\put(-10,-20){$\bullet$} \put(-20,-15){$z_2$}

 \put(-8,-18){\line(-1,2){10}} \put(-8,-18){\line(-1,-2){10}}
  \put(-6,-18){\line(4,3){28}}  \put(-6,-18){\line(4,-3){28}}
   \put(-8,-18){\line(-1,0){30}} \put(-40,-20){$\bullet$}

\put(40,-20){$\bullet$}

\put(-18,-65){Figure 7.}
\end{picture}
\begin{picture}(0,0)(-50,21)

\put(5,20){$\circ$} \put(44,20){$\blacktriangle$}
\put(5,-20){$\bullet$} \put(45,-20){$\bullet$} \put(25,-30){$\circ$}

\put(8,21){\line(0,-1){39}} \put(48,21){\line(0,-1){37}}
\put(8,-17){\line(1,1){38}} \put(48,-17){\line(-1,1){39}}
\put(29,-27){\line(2,1){18}} \put(26,-27){\line(-2,1){18}}
\put(28,-27){\line(2,5){19}} \put(48,23){\line(-1,0){39}}

\put(25,-50){$\star$} \put(25,-58){$z_1$}

\put(27,-48){\line(0,1){20}} \put(27,-48){\line(2,3){20}}
\put(27,-48){\line(-2,3){20}}

\put(27,-48){\line(1,0){40}} \put(67,-48){\line(0,1){70}}
\put(67,22){\line(-1,0){20}}

\put(27,-48){\line(-1,0){40}} \put(-13,-48){\line(0,1){70}}
\put(-13,22){\line(1,0){19}}

\put(0,-78){Figure 8.}
\end{picture}
\end{center}
$$$$\\\\\\\\\\\\\\
\textbf{Case C.} If the drawing $\phi(\langle E_{XY}\cup E_{XU}\cup
E_{YU}\rangle)$ contains a unique region, $f$, which contains more
than 4 vertices, then, by Lemma \ref{35}, the only possible drawings
of $\phi(\langle E_{XY}\cup E_{XU}\cup E_{YU}\rangle)$ are shown in
Figure 3. If $\phi(\langle E_{XY}\cup E_{XU}\cup E_{YU}\rangle)=D_1$
or $D_3$, then $cr_\phi(E)\geq 3$. If $\phi(\langle E_{XY}\cup
E_{XU}\cup E_{YU}\rangle)=D_2$, we may assume that there exists
$z_i$, say $z_1$, such that $\phi(z_i)$ lies in $f$, that is the
unbounded region of $D_2$ and $cr_\phi(E(z_i),E_{XY}\cup E_{XU}\cup
E_{YU})=0$ (otherwise, $cr_\phi(E)\geq 3$). Then $\phi(\langle
E(z_1)\cup E_{XY}\cup E_{XU}\cup E_{YU}\rangle)$ must be drawn as in
Figure 8.

However, no matter which region $z_2$ is placed, we have
$cr_\phi(E(z_2),E(z_1)\cup E_{XY}\cup E_{XU}\cup E_{YU})\geq 1$
which implies that $cr_\phi(E)\geq 3$. Therefore (\ref{26}) is true
for this case.

This shows that (\ref{26}) is true for $n=2$. Now consider $n\geq
3$. Suppose (\ref{26}) is true for all value less than $n$ and is
not true for $n$. Then there exists a good immersion $\phi$ of
$K_{1,2,2,n}$ such that
\begin{eqnarray}\label{27}
cr_{\phi}(E)<Z(5,n)+\lfloor\frac{3n}{2}\rfloor.
\end{eqnarray} Let $W=E_{XY}\cup E_{XU} \cup E_{YU}$.

Note that (\ref{5}) is also true for $\phi$. Since
$\langle\displaystyle\bigcup_{i=1}^n E(z_i)\rangle\cong K_{5,n}$, by
(\ref{21}), we have
\begin{eqnarray}\label{29}
cr_\phi(\bigcup_{i=1}^n E(z_i))\geq Z(5,n).
\end{eqnarray}
If $cr_\phi({W},E(z_i))\geq 2$ for all $i$, by (\ref{5}) and
(\ref{29}), we have $cr_\phi(E)\geq Z(5,n)+2n$ which contradict to
(\ref{27}). Therefore $cr_\phi({W},E(z_i))\leq 1$ for some
$i$. We will consider two cases:\\

\textbf{Case 1.} \emph{If $cr_\phi({W},E(z_i))=0$ for some $i$.}\\

By reordering, we may assume that $cr_\phi({W},E(z_1))=0.$
$\phi(\langle W\rangle)$ divides $\mathbb{R}^2$ into regions and the
condition $cr_\phi({W},E(z_1))=0$ implies that $\phi(X\cup Y\cup U)$
is contained in the boundary of one of the regions. Denote $F=W\cup
E(z_1)$. Then Figure 9 has shown all the possible drawings of $\phi|
\langle F\rangle$.

\begin{center}
\begin{picture}(0,0)(140,25)

\put(-20,0){$\circ$} \put(20,0){$\bullet$} \put(-20,-40){$\circ$}
\put(20,-40){$\bullet$}
 \put(18,-20){$\blacktriangle$} \put(59,-20){$\star$} \put(68,-20){$z_1$}

\put(22,2){\line(0,-1){40}} \put(-18,2){\line(1,0){40}}
\put(-18,2){\line(1,-1){40}} \put(22,2){\line(-1,-1){40}}
\put(22,-38){\line(-1,0){40}}

\put(22,-18){\line(-2,1){40}} \put(22,-18){\line(-2,-1){40}}
\put(22,-17){\line(1,0){39}}

\put(61,-17){\line(-2,1){37}} \put(61,-17){\line(-2,-1){37}}
\put(61,-17){\line(-1,-1){40}} \put(21,-57){\line(-1,0){38}}
\put(-17,-57){\line(0,1){20}} \put(61,-17){\line(-1,1){40}}
\put(21,23){\line(-1,0){38}} \put(-17,23){\line(0,-1){20}}

\put(50,0){$I$} \put(0,5){$II$} \put(0,-50){$III$}
\put(25,-15){$IV$} \put(28,-30){$V$}

 \put(0,-70){$D_1$}
\end{picture}
\begin{picture}(0,0)(25,25)
\put(-20,0){$\bullet$} \put(20,0){$\circ$} \put(-21,-40){$\circ$}
\put(20,-40){$\bullet$}
 \put(18,-20){$\blacktriangle$} \put(59,-20){$\star$} \put(68,-20){$z_1$}

\put(22,2){\line(0,-1){40}} \put(-18,2){\line(1,0){40}}
\put(-18,2){\line(0,-1){40}} \put(22,2){\line(-1,-2){10}}
\put(22,-38){\line(-1,2){10}} \put(22,-38){\line(-1,0){40}}

\put(22,-18){\line(-2,1){40}} \put(22,-18){\line(-2,-1){40}}
\put(22,-17){\line(1,0){39}}

\put(61,-17){\line(-2,1){37}} \put(61,-17){\line(-2,-1){37}}
\put(61,-17){\line(-1,-1){40}} \put(21,-57){\line(-1,0){38}}
\put(-17,-57){\line(0,1){20}} \put(61,-17){\line(-1,1){40}}
\put(21,23){\line(-1,0){38}} \put(-17,23){\line(0,-1){20}}

\put(50,0){$I$} \put(0,5){$II$} \put(0,-50){$III$}
\put(25,-15){$IV$} \put(28,-30){$V$}

 \put(0,-70){$D_2$}

 \put(-10,-95){Figure 9.}
\end{picture}
\begin{picture}(0,0)(-90,25)

\put(-20,0){$\bullet$} \put(20,0){$\bullet$}
\put(-20,-40){$\circ$} \put(20,-40){$\circ$}
 \put(18,-20){$\blacktriangle$} \put(59,-20){$\star$} \put(68,-20){$z_1$}

\put(22,2){\line(0,-1){40}} \put(-18,2){\line(1,-1){40}}
\put(22,2){\line(-1,-1){40}}

\put(-18,2){\line(0,-1){40}}

\put(22,2){\line(-1,-2){10}} \put(22,-38){\line(-1,2){10}}

\put(22,-18){\line(-2,1){40}} \put(22,-18){\line(-2,-1){40}}
\put(22,-17){\line(1,0){39}}

\put(61,-17){\line(-2,1){37}} \put(61,-17){\line(-2,-1){37}}
\put(61,-17){\line(-1,-1){40}} \put(21,-57){\line(-1,0){38}}
\put(-17,-57){\line(0,1){20}} \put(61,-17){\line(-1,1){40}}
\put(21,23){\line(-1,0){38}} \put(-17,23){\line(0,-1){20}}

\put(50,0){$I$} \put(0,5){$II$} \put(0,-50){$III$}
\put(25,-15){$IV$} \put(28,-30){$V$}

 \put(0,-70){$D_3$}
\end{picture}
\end{center}
$$$$\\\\\\\\\\\\\\\\

For each drawing $D_i$ ($1\leq i\leq 3$), if $\phi(z_j)$ ($2\leq
j\leq n$) is located in the region $I$ to $V$, we have
\begin{eqnarray}\label{30}
cr_\phi(F,E(z_j))\geq 4.
\end{eqnarray}
If $\phi(z_j)$ ($2\leq j\leq n$) is located in the region other than
$I$ to $V$, we have
\begin{eqnarray}\label{31}
cr_\phi(W,E(z_j))\geq 3.
\end{eqnarray}

Let $l$ be the number of $\phi(z_i)$ ($2\leq i\leq n$) such that it
is not located in the region $I$ to $V$. Note that $cr_\phi(W)\geq
2$ and combining (\ref{5}), (\ref{27}), (\ref{29}) and (\ref{31}),
we have
\begin{eqnarray}\label{32}
l\leq\lfloor\frac{n-2}{2}\rfloor.
\end{eqnarray}
Note that the number of $\phi(z_i)$ ($2\leq i\leq n$) such that it
is located in the region $I$ to $V$ is $n-l-1$. Note that (\ref{10})
is also true for $\phi$. Note also that
$\displaystyle\bigcup_{i=2}^n E(z_i)\cong K_{5,n-1}$, by (\ref{21}),
we have
\begin{eqnarray}\label{34}
cr_\phi(\bigcup_{i=2}^n E(z_i))\geq Z(5,n-1).
\end{eqnarray}
Since $cr_\phi(F)\geq 2$ in $D_i$ ($1\leq i\leq 3$),  by (\ref{10}),
(\ref{30}), (\ref{31}), (\ref{32}) and (\ref{34}), we have
$cr_\phi(E)\geq 2+Z(5,n-1)+4(n-l-1)+3l\geq
Z(5,n)+\displaystyle\lfloor\frac{3n}{2}\rfloor$ which contradicts to
(\ref{27}).\\

\textbf{Case 2.} \emph{If $cr_\phi({W},E(z_i))\geq 1$ for all $i$.}\\

Since $cr_\phi({W},E(z_i))\leq 1$ for some $i$,
$cr_\phi({W},E(z_i))=1$ for some $i$. Now by reordering, we may
assume that $cr_\phi({W},E(z_1))=1$. Then there is at least one
region in $\phi(\langle W\rangle)$ containing at least 4 vertices
in $\phi(X\cup Y\cup U)$. Then, by Lemma \ref{24}, $\phi(\langle
W\rangle)$ has a unique region whose boundary contains at least 4
vertices. We denote the unique region by $f$ and its boundary by
$\partial f$.

Suppose $\partial f$ contains more than 4 vertices in $\phi(X\cup
Y\cup U)$. Then by Lemma \ref{35}, there are 3 possible drawings of
$\phi(\langle W\rangle)$, which are shown in Figure 3 and $f$ is the
unbounded region for each $D_i$. Let $F=W\cup E(z_1)$. Under the
condition $cr_\phi({W},E(z_1))=1$, $\phi(z_1)$ must be drawn in $f$.
However, if $\phi(z_1)$ is drawn in $f$, $cr_\phi({W},E(z_1))\neq
1$.

Therefore, $\partial f$ must contain exactly 4 vertices in
$\phi(X\cup Y\cup U)$. For the $z_i$ such that
$cr_\phi(W,E(z_i))=1$, $\phi(z_i)$ must lie in $f$. Suppose the
vertices in $\phi(X\cup Y\cup U)$ contained in $\partial f$, in the
clockwise manner, are $a_1, a_2, a_3, a_4$ and $a_5$ is the vertex
of $\phi(X\cup Y\cup U)$ which is not contained in $f$. For
$cr_\phi({W},E(z_i))=1$, $\phi(z_i)$ is drawn in $f$. Since
$\phi(z_ia_5)$ cross $\partial f$, $\phi(z_ia_j)$, $1\leq j\leq 4$,
does not cross. See Figure 10 for the case $\phi(z_ia_5)$ cross the
boundary of $f$ between $a_1$ and $a_2$. Note that the boundary of
$f$ between $a_j$ and $a_{j+1}$ ($1\leq j\leq 4$ and mod 4 for
$j+1$) may not be the image of a single edge of $W$, but it must be
composed of some portions of the image of the edges of $W$ under
$\phi$.

\begin{center}
\begin{picture}(0,0)(70,5)
\put(20,0){$\bullet$} \put(25,5){$a_1$}
 \put(20,-40){$\bullet$} \put(25,-45){$a_2$}
 \put(-20,-40){$\bullet$} \put(-30,-45){$a_3$}
 \put(-20,0){$\bullet$} \put(-30,5){$a_4$}

\put(22,2){\line(0,-1){40}} \put(22,2){\line(-1,0){40}}
\put(22,-38){\line(-1,0){40}} \put(-18,-38){\line(0,1){40}}

\put(22,2){\line(-1,-1){20}} \put(-18,2){\line(1,-1){40}}
\put(-18,-38){\line(1,1){20}}

\put(2,-18){\line(1,0){40}} \put(0,-20){$\bullet$}
\put(10,-16){$z_i$}

\put(40,-20){$\bullet$} \put(48,-20){$a_5$}

\put(-18,-65){Figure 10.}
\end{picture}
\begin{picture}(0,0)(-70,5)
\put(20,0){$\bullet$} \put(25,5){$a_1$}
 \put(20,-40){$\bullet$} \put(25,-45){$a_2$}
 \put(-20,-40){$\bullet$} \put(-30,-45){$a_3$}
 \put(-20,0){$\bullet$} \put(-30,5){$a_4$}

\put(22,2){\line(0,-1){40}} \put(22,2){\line(-1,0){40}}
\put(22,-38){\line(-1,0){40}} \put(-18,-38){\line(0,1){40}}

\put(22,2){\line(-1,-1){20}} \put(-18,2){\line(1,-1){40}}
\put(-18,-38){\line(1,1){20}}

\put(2,-18){\line(1,0){40}} \put(0,-20){$\bullet$}
\put(8,-16){$z_1$}

\put(-10,-20){$\bullet$}

 \put(-8,-18){\line(-1,2){10}} \put(-8,-18){\line(-1,-2){10}}
  \put(-6,-18){\line(4,3){28}}  \put(-6,-18){\line(4,-3){28}}
   \put(-8,-18){\line(-1,0){30}} \put(-40,-20){$\bullet$}
   \put(-50,-20){$a_5$}

\put(40,-20){$\bullet$} \put(48,-20){$a_5$}

\put(-18,-70){Figure 11.}
\end{picture}
\end{center}
$$$$\\\\\\\\

Let $l$ be the number of $z_i$ such that $cr_\phi({W},E(z_i))=1$.
Also, let $l_j$ ($1\leq j\leq 4$) be the number of $z_i$ such that
$cr_\phi({W},E(z_i))=1$ and $\phi(z_ia_5)$ crosses the boundary of
$f$ between $a_j$ and $a_{j+1}$ (mod 4 for $j+1$) where $1\leq j\leq
4$. Then we have $\displaystyle\sum_{j=1}^4l_j=l$. We may assume
$l_1+l_3\geq l_2+l_4$. Also we may assume $l_1\geq l_3$.

Now by reordering, we may assume that $cr_\phi({W},E(z_1))=1$ such
that $\phi(z_1a_5)$ crosses the boundary of $f$ between $a_1$ and
$a_{2}$. Let $F=W\cup E(z_1)$, then $\phi|F$ must be drawn as in
Figure 10. For $z_i\neq z_1$ such that $cr_\phi({W},E(z_i))=1$, if
$\phi(z_ia_5)$ crosses the boundary between $a_1$ and $a_2$, then
$cr_\phi(F,E(z_i))\geq 5$; if $\phi(z_ia_5)$ crosses the boundary
between $a_2$ and $a_3$ or the boundary between $a_4$ and $a_1$, we
have $cr_\phi(F,E(z_i))\geq 4$; if $\phi(z_ia_5)$ crosses the
boundary between $a_3$ and $a_4$ (see Figure 11), we have
$cr_\phi(F,E(z_i))\geq 3$.

Since $l$ is the number of $z_i$ where $1\leq i\leq n$ such that
$cr_\phi(W,E(z_i))=1$, the number of $z_i$ where $1\leq i\leq n$
such that $cr_\phi(W,E(z_i))\geq 2$ is $n-l$. From (\ref{5}),
(\ref{27}) and (\ref{29}), we have
\begin{eqnarray}\label{36}
l\geq\lceil\frac{n}{2}\rceil+1.
\end{eqnarray}

We will consider two subcases:

\textbf{Case A. }Suppose for all $z_i$ ($2\leq i\leq n$) such that
$cr_\phi(W,E(z_i))=2$ we have
\begin{eqnarray}\label{37}
cr_\phi(F,E(z_i))\geq 3.
\end{eqnarray}
Note that (\ref{10}) is also true for $\phi$. Then by (\ref{10}),
(\ref{34}), (\ref{37}), the fact that $cr_\phi(F)\geq 1$ and our
previous discussion, we have
\begin{eqnarray*}
\begin{array}{rcl}
cr_\phi(E)&\geq& 1+Z(5,n-1)\\
&&+\displaystyle\sum_{cr_\phi(W,E(z_i))=1}cr_\phi(F,E(z_i))
+\displaystyle\sum_{cr_\phi(W,E(z_i))\geq 2}cr_\phi(F,E(z_i))\\
&\geq& 1+Z(5,n-1)+3(n-l)+5(l_1-1)+4(l_2+l_4)+3l_3\\
&\geq& Z(5,n-1)+3n+l-4\\
&\geq& Z(5,n)+\displaystyle\lfloor\frac{3n}{2}\rfloor
\end{array}
\end{eqnarray*}
where the third inequality follows from the fact that $l_1\geq l_3$
and $\displaystyle\sum_{j=1}^4l_j=l$ and the last inequality follows
from (\ref{36}). However, $cr_\phi(E)\geq
Z(5,n)+\displaystyle\lfloor\frac{3n}{2}\rfloor$ contradicts to
(\ref{27}).

\textbf{Case B.} Suppose there exists $z_i$ ($2\leq i\leq n$) such
that $cr_\phi(F,E(z_i))=cr_\phi(W,E(z_i))=2$. We may assume
$cr_\phi(F,E(z_2))=cr_\phi(W,E(z_2))=2$ which implies that
\begin{eqnarray}\label{106}
cr_\phi(E(z_1),E(z_2))=0.
\end{eqnarray}
For the vertices $z_i, z_j, z_k\in Z$, $\langle E(z_i)\cup
E(z_j)\cup E(z_k)\rangle$ is homeomorphic to $K_{5,3}$. Hence by
(\ref{3}) and (\ref{4}), we have
\begin{eqnarray}\label{107}
cr_\phi(E(z_i)\cup E(z_j),E(z_k))+cr_\phi(E(z_i),E(z_j))\geq 4.
\end{eqnarray}
Therefore, by using (\ref{106}) and (\ref{107}), we have
\begin{eqnarray}\label{108}
cr_\phi(E(z_1)\cup E(z_2),E(z_k))\geq 4.
\end{eqnarray}
Let $E'=E-(E(z_1)\cup E(z_2)$. Then $\langle
E'\rangle=K_{1,2,2,n-2}$ and
\begin{eqnarray}\label{110}
\begin{array}{rcl}
cr_\phi(E)&=&cr_\phi(E')+cr_\phi(E(z_1)\cup
E(z_2))\\
&&+cr_\phi(W,E(z_1)\cup E(z_2))+\displaystyle\sum_{i=3}^n
cr_\phi(E(z_1)\cup E(z_2),E(z_i)).
\end{array}
\end{eqnarray}
On the other hand, if $cr_\phi(E(z_1),E(z_2))=0$, we have
\begin{eqnarray}\label{111}
cr_\phi(W,E(z_1)\cup E(z_2))\geq 3.
\end{eqnarray}
To see this, note that under the condition
$cr_\phi(E(z_1),E(z_2))=0$, $\phi(\langle E(z_1)\cup E(z_2)\rangle)$
must be drawn as in Figure 12 where $\blacktriangle$ is the vertex
in $X$, $\bullet$ are the vertices in $Y$ and $\circ$ are the
vertices in $U$.

\begin{center}
\begin{picture}(0,0)(100,5)
\put(0,0){$\star$} \put(0,-60){$\star$}

\put(7,3){$z_1$} \put(7,-63){$z_2$}

 \put(10,-30){$\bullet$} \put(-10,-30){$\bullet$}
 \put(30,-30){$\circ$} \put(-30,-30){$\circ$}
 \put(-50,-30){$\blacktriangle$}

\put(2,2){\line(1,-3){10}} \put(2,2){\line(-1,-3){10}}
\put(2,2){\line(1,-1){30}} \put(2,2){\line(-1,-1){30}}
 \put(2,2){\line(-5,-3){50}}

\put(2,-58){\line(1,3){10}} \put(2,-58){\line(-1,3){10}}
\put(2,-58){\line(1,1){30}} \put(2,-58){\line(-1,1){30}}
\put(2,-58){\line(-5,3){50}}

\put(0,-75){$F_1$}
\end{picture}
\begin{picture}(0,0)(-0,5)
\put(0,0){$\star$} \put(0,-60){$\star$}

\put(7,3){$z_1$} \put(7,-63){$z_2$}

 \put(10,-30){$\bullet$} \put(-10,-30){$\circ$}
 \put(30,-30){$\circ$} \put(-30,-30){$\bullet$}
 \put(-50,-30){$\blacktriangle$}

\put(2,2){\line(1,-3){10}} \put(2,2){\line(-1,-3){10}}
\put(2,2){\line(1,-1){30}} \put(2,2){\line(-1,-1){30}}
 \put(2,2){\line(-5,-3){50}}

\put(2,-58){\line(1,3){10}} \put(2,-58){\line(-1,3){10}}
\put(2,-58){\line(1,1){30}} \put(2,-58){\line(-1,1){30}}
 \put(2,-58){\line(-5,3){50}}

\put(0,-75){$F_2$}

\put(-15,-90){Figure 12.}
\end{picture}
\begin{picture}(0,0)(-100,5)
\put(0,0){$\star$} \put(0,-60){$\star$}

\put(7,3){$z_1$} \put(7,-63){$z_2$}

 \put(10,-30){$\circ$} \put(-10,-30){$\bullet$}
 \put(30,-30){$\circ$} \put(-30,-30){$\bullet$}
 \put(-50,-30){$\blacktriangle$}

\put(2,2){\line(1,-3){10}} \put(2,2){\line(-1,-3){10}}
\put(2,2){\line(1,-1){30}} \put(2,2){\line(-1,-1){30}}
 \put(2,2){\line(-5,-3){50}}

\put(2,-58){\line(1,3){10}} \put(2,-58){\line(-1,3){10}}
\put(2,-58){\line(1,1){30}} \put(2,-58){\line(-1,1){30}}
 \put(2,-58){\line(-5,3){50}}

\put(0,-75){$F_3$}

\end{picture}
\end{center}
$$$$\\\\\\\\\\\\

Naming the vertices of $\phi(\langle E(z_1)\cup E(z_2)\rangle)$ in
Figure 12 from left to right $a_1$, $a_2,$..., $a_5$. In $F_1$,
$a_1a_3$, $a_1a_4$, $a_2a_4$, $a_3a_5$ have to cross with the edges
in $W$ at least once. Therefore, $cr_\phi(W,E(z_1)\cup E(z_2))\geq
4$. In $F_2$, $a_1a_3$, $a_1a_4$, $a_2a_5$ have to cross with the
edges in $W$ at least once. Therefore, $cr_\phi(W,E(z_1)\cup
E(z_2))\geq 3$. In $F_3$, $a_1a_3$, $a_1a_4$, $a_2a_4$, $a_2a_5$,
$a_3a_5$ have to cross with the edges in $W$ at least once.
Therefore, $cr_\phi(W,E(z_1)\cup E(z_2))\geq 3$. This proves
(\ref{111}).

Then, by the induction assumption and (\ref{108}), (\ref{110}),
(\ref{111}), $ cr_\phi(E)\geq
Z(5,n-2)+\displaystyle\lfloor\frac{3(n-2)}{2}\rfloor+3+4(n-2)\geq
Z(5,n)+\lfloor\frac{3n}{2}\rfloor. $ which contradicts to
(\ref{27}).
\end{proof}

\section{Crossing number of $K_{1,1,1,2,n}$}

In this section, the crossing number of $K_{1,1,1,2,n}$ is found.
More precisely, we have

\begin{thm}\label{Z1}
The crossing number of the complete 5-partite graph
$K_{1,1,1,2,n}$ is given by
$$cr(K_{1,1,1,2,n})=Z(5,n)+2n.$$
\end{thm}
\begin{proof}
Let $(X,Y,S,T,Z)$ be partition of $K_{1,1,1,2,n}$ such that $X=\{
x_1\}$, $Y=\{ y_1\}$, $S=\{ s_1\}$, $T=\{t_1, t_2\}$ and
$Z=\displaystyle\bigcup_{i=1}^n\{z_i\}$. To show that
$cr(K_{1,1,1,2,n})\leq Z(5,n)+2n,$ see Figure 13 for $n=4$, and it
can be generalized to all $n$.

\begin{center}
\begin{picture}(60,20)(-30,35)
\put(0,0){$\bullet$} \put(1,40){$\bullet$}
\put(-11,-40){$\bullet$} \put(0,-81){$\bullet$}
\put(0,-121){$\bullet$}

\put(-7,7){$t_1$} \put(-13,41){$y_1$} \put(-4,-45){$x_1$}
\put(-12,-81){$t_2$} \put(10,-121){$s_1$}

 \put(3,42){\line(1,-6){10}} \put(13,-17){\line(-1,-1){20}}

 \put(2,2){\line(-1,-4){10}}
\put(-8,-38){\line(1,-4){10}} \put(3,2){\line(0,1){40}}

\put(20,-20){$\bullet$} \put(60,-22){$\bullet$}
\put(-40,-20){$\bullet$} \put(-80,-21){$\bullet$}

\put(28,-18){$z_1$} \put(-49,-20){$z_3$} \put(67,-20){$z_2$}
\put(-89,-21){$z_4$}

\put(22,-18){\line(-1,1){20}} \put(24,-18){\line(-1,3){20}}
\put(24,-17){\line(-3,-2){32}} \put(24,-18){\line(-1,-3){20}}

\put(62,-18){\line(-3,1){60}} \put(64,-18){\line(-1,1){60}}
 \put(64,-20){\line(-4,-1){71}} \put(64,-18){\line(-1,-1){60}}

\put(-38,-18){\line(2,1){40}} \put(-38,-18){\line(2,3){40}}
\put(-38,-18){\line(3,-2){28}} \put(-38,-18){\line(2,-3){40}}

\put(-78,-18){\line(4,1){80}} \put(-78,-17){\line(4,3){80}}
\put(-78,-20){\line(4,-1){70}} \put(-78,-20){\line(4,-3){80}}

\put(95,-18){\line(-3,2){90}} \put(95,-18){\line(-3,-2){90}}

\put(4,-119){\line(1,5){20}} \put(3,-119){\line(0,1){40}}
\put(4,-119){\line(3,5){60}} \put(2,-119){\line(-2,5){40}}
\put(2,-119){\line(-2,5){40}} \put(1,-119){\line(-4,5){80}}

\put(-122,-17){\line(6,1){123}} \put(-122,-17){\line(6,-5){123}}

 \put(-10,-38){\line(1,-6){13}}

 \put(4,-119){\line(3,2){153}} \put(157,-17){\line(-5,2){153}}

 \put(-15,-147){Figure 13.}
\end{picture}
\end{center}
$$$$\\\\\\\\\\\\\\\\\\\\\\\\\\

Therefore it is sufficient to prove
that
\begin{eqnarray}\label{Z2}
cr(K_{1,1,1,2,n})\geq Z(5,n)+2n.
\end{eqnarray}
We will prove (\ref{Z2}) by induction on $n$. For $n=1$, note that
$\langle x_1\cup y_1\cup s_1\cup t_1\cup z_1 \rangle$ is isometric
to $K_5$ and it is clear that $K_{5}$ is nonplanar, therefore
$cr(K_{1,1,1,2,1})\geq 1$. Suppose that $cr(K_{1,1,1,2,1})<2$, then
$cr(K_{1,1,1,2,1})=1$. Let $\phi$ be an immersion of $K_{1,1,1,2,1}$
such that its crossing number is 1. Therefore $cr_\phi(\langle
x_1\cup y_1\cup s_1\cup t_1\cup z_1 \rangle)=1$. It is well-known
that there exists a unique drawing of $K_5$ such that its crossing
number is 1, namely, the drawing in Figure 14.

However, under the condition that $cr(K_{1,1,1,2,1})=1$, we must
have $cr_\phi(E(t_2))\leq 1$. This implies that there is a face in
the drawing of $K_5$ such that there is a face containing at least 4
vertices on its boundary, which is impossible. Therefore, (\ref{Z2})
is true for $n=1$.

\begin{center}
\begin{picture}(0,0)(127,10)
\put(0,0){$\bullet$} \put(40,0){$\bullet$} \put(40,-40){$\bullet$}
\put(0,-40){$\bullet$}

\put(2,3){\line(1,0){40}} \put(2,3){\line(1,-1){40}}
\put(2,3){\line(0,-1){40}}

\put(42,-37){\line(-1,0){40}} \put(42,-37){\line(0,1){40}}

\put(2,-37){\line(-3,1){45}} \put(42,3){\line(-3,1){45}}
\put(-3,18){\line(-1,-1){40}}

\put(10,-30){$\bullet$}

\put(12,-28){\line(-1,-1){10}} \put(12,-28){\line(1,1){30}}
\put(13,-28){\line(-1,3){10}} \put(13,-27){\line(3,-1){30}}

 \put(-1,-70){Figure 14.}
\end{picture}
\begin{picture}(0,0)(60,30)
\put(0,20){$\circ$} \put(40,20){$\bullet$} \put(0,-20){$\bullet$}
\put(40,-20){$\circ$} \put(20,0){$\blacktriangle$}

\put(3,21){\line(0,-1){39}} \put(4,23){\line(1,0){39}}
\put(43,21){\line(0,-1){37}} \put(4,-17){\line(1,0){37}}
\put(3,-17){\line(1,1){19}} \put(42,21){\line(-1,-1){18}}

\put(23,2){\line(1,-1){18}} \put(22,3){\line(-1,1){18}}

\put(0,-50){Figure 15.}
\end{picture}
\begin{picture}(0,0)(-25,30)
\put(0,20){$\circ$} \put(40,20){$\bullet$} \put(0,-20){$\bullet$}
\put(40,-20){$\circ$} \put(20,0){$\blacktriangle$}

\put(3,21){\line(0,-1){39}} \put(4,23){\line(1,0){39}}
\put(43,21){\line(0,-1){37}} \put(4,-17){\line(1,0){37}}
\put(3,-17){\line(1,1){19}} \put(42,21){\line(-1,-1){18}}

\put(23,2){\line(1,-1){18}} \put(22,3){\line(-1,1){18}}

\put(-20,0){$\star$}

\put(-18,2){\line(1,0){40}} \put(-18,2){\line(1,1){20}}
\put(-18,2){\line(1,-1){20}} \put(-18,2){\line(0,1){30}}
\put(-18,32){\line(1,0){60}} \put(42,32){\line(0,-1){10}}
\put(-18,2){\line(0,-1){30}} \put(-18,-28){\line(1,0){60}}
\put(42,-28){\line(0,1){10}}

\put(-10,-50){Figure 16.}
\end{picture}
\begin{picture}(0,0)(-110,10)
\put(20,0){$\bullet$} \put(25,5){$a_1$}
 \put(20,-40){$\bullet$} \put(25,-45){$a_2$}
 \put(-20,-40){$\bullet$} \put(-30,-45){$a_3$}
 \put(-20,0){$\bullet$} \put(-30,5){$a_4$}

\put(22,2){\line(0,-1){40}} \put(22,2){\line(-1,0){40}}
\put(22,-38){\line(-1,0){40}} \put(-18,-38){\line(0,1){40}}

\put(22,2){\line(-1,-1){20}} \put(-18,2){\line(1,-1){40}}
\put(-18,-38){\line(1,1){20}}

\put(2,-18){\line(1,0){40}} \put(0,-20){$\bullet$}
\put(10,-16){$x_1$}

\put(40,-20){$\bullet$} \put(48,-20){$a_5$}

\put(-18,-69){Figure 17.}
\end{picture}
\end{center}
$$$$\\\\\\\\

For $n=2$, suppose that (\ref{Z2}) is not true, that is,
$cr(K_{1,1,1,2,2})<4$. Then there exists a good immersion $\phi$ of
$K_{1,1,1,2,2}$ such that $cr_\phi(E)<4.$ Let $Q=E_{ST}\cup
E_{SZ}\cup E_{TZ}$. Since $\langle Q\rangle\cong K_{1,2,2}$, by
Lemma \ref{24}, $\phi(\langle Q\rangle)$ has at most one region
whose boundary contains at least 4 vertices in $\phi(S\cup T\cup
Z)$. In $\phi(\langle Q\rangle)$, if all the regions whose boundary
contains less than 4 vertices in $\phi(S\cup T\cup Z)$, then we have
$cr_\phi(E(x_1),Q)\geq 2$ and $cr_\phi(E(y_1),Q)\geq 2$ which
implies that $cr_\phi(E)\geq 4$. Therefore $\phi(\langle Q\rangle)$
has a unique region whose boundary contains at least 4 vertices in
$\phi(S\cup T\cup Z)$. Denote the unique face by $f$ and it boundary
by $\partial f$. We will consider three cases:

\textbf{Case A.} If $cr_\phi(Q)=0$, since $K_{1,2,2}$ is
3-connected, it has a unique embedding in the plane, namely, the
drawing in Figure 15.

 If $\phi(x_1)$ does not lie in $f$, then
$cr_\phi(E(x_1),Q)\geq 2$. If $\phi(x_1)$ lie in $f$, then
$cr_\phi(E(x_1),Q)\geq 1$. Since $cr_\phi(E)<4$,
$cr_\phi(E(x_1),Q)=1$ or $cr_\phi(E(y_1),Q)=1$. Without loss of
generality, we may assume that $cr_\phi(E(x_1),Q)=1$ which implies
that $\phi(x_1)$ lies in $f$. Denote $K=E(x_1)\cup Q$, then the only
possible drawing of $\phi(\langle K\rangle)$ is shown in Figure 16.
However, no matter which region $\phi(y_1)$ is contained in will
result in $cr_\phi(E)\geq 4$.

\textbf{Case B.} Suppose that $cr_\phi(Q)\geq 1$ and $\partial f$
contains more than 4 vertices in $\phi(S\cup T\cup Z)$. By Lemma
\ref{35}, the possible drawings of $\phi(\langle Q\rangle)$ are
shown in Figure 3. Note that $f$ is just the unbounded region.

If $\phi(x_1)$ does not lie in $f$, then $cr_\phi(E(x_1),Q)\geq 3$.
Since $cr_\phi(E)<4$, $cr_\phi(E(x_1), Q)\leq 1$ or
$cr_\phi(E(y_1),Q)\leq 1$. Without loss of generality, we may assume
that $cr_\phi(E(x_1),Q)\leq 1$ which implies that $\phi(x_1)$ lies
in $f$. If $\phi(x_1)$ lies in $f$, under the condition
$cr_\phi(E(x_1),Q)\leq 1$, we have $cr_\phi(E(x_1),Q)=0$. Denote
$K=E(x_1)\cup Q$, then the possible drawings of $\phi(\langle
K\rangle)$ are shown in Figure 9. However, no matter which region
$\phi(y_1)$ lies will result in $cr_\phi(E)\geq 4$, contradicted to
$cr_\phi(E)<4$.

\textbf{Case C.} Suppose that $cr_\phi(Q)\geq 1$ and $\partial f$
contains exactly 4 vertices in $\phi(S\cup T\cup Z)$. Since
$cr_\phi(E)\leq 3$, then $cr_\phi(E(x_1), Q)\leq 1$ or
$cr_\phi(E(y_1), Q)\leq 1$ which implies that $\phi(x_1)$ or
$\phi(y_1)$ lies in $f$. We may assume that $\phi(x_1)$ lies in
$f$. Note that if $\phi(x_1)$ lies in $f$, then $cr_\phi(E(x_1),
Q)\geq 1$. Therefore $cr_\phi(E(x_1), Q)=1$. On the other hand, if
$\phi(y_1)$ does not lie in $f$, then $cr_\phi(E(y_1), Q)\geq 2$
which implies $cr_\phi(E)\geq 4$. Thus, both $\phi(x_1)$ and
$\phi(y_1)$ lie in $f$.

Suppose the vertices in $\phi(S\cup T\cup Z)$ contained in $\partial
f$, in the clockwise manner, are $a_1, a_2, a_3, a_4$ and $a_5$ is
the vertex of $\phi(S\cup T\cup Z)$ which is not contained in $f$.
Since $\phi(x_1a_5)$ must cross $\partial f$, under the condition
$cr_\phi(E(x_1), Q)=1$, $\phi(x_1a_j)$ where $1\leq j\leq 4$, does
not cross the edges in $Q$. See Figure 17 for the case
$\phi(x_1a_5)$ cross the boundary of $f$ between $a_1$ and $a_2$.
Note that the boundary of $f$ between $a_j$ and $a_{j+1}$ ($1\leq
j\leq 4$ and mod 4 for $j+1$) may not be the image of a single edge
of $Q$, but it must be composed of some portions of the image of the
edges of $Q$ under $\phi$.

Now by reordering, we may suppose $\phi(x_1a_5)$ cross the boundary
of $f$ between $a_1$ and $a_2$. However, if $\phi(y_1)$ lies in $f$,
$cr_\phi(E(x_1)\cup Q, E(y_1))\geq 2$. This gives $cr_\phi(E)\geq
4$, contradicted to $cr_\phi(E)<4$.

This shows that (\ref{Z2}) is true for $n=2$. Now we can assume that
$n\geq 3$. Suppose (\ref{Z2}) is true for all value less than $n$
and is not true for $n$. Then there exists a good immersion $\phi$
of $K_{1,1,1,2,n}$ such that
\begin{eqnarray}\label{Z3}
cr_{\phi}(E)<Z(5,n)+2n.
\end{eqnarray} Let $W=E_{XY}\cup E_{XS}\cup E_{XT} \cup E_{YS}\cup E_{YT}\cup E_{ST}$.
Note that (\ref{5}) and (\ref{29}) is also true for $\phi$.

If $cr_\phi({W},E(z_i))\geq 2$ for all $i$, by (\ref{5}) and
(\ref{29}), we have $cr_\phi(E)\geq Z(5,n)+2n$ which contradict to
(\ref{Z3}). Therefore $cr_\phi({W},E(z_i))\leq 1$ for some $i$. We
will consider the following two cases:
\begin{description}
  \item[Case 1.] $cr_\phi({W},E(z_i))=0$ for some $i$;
  \item[Case 2.] $cr_\phi({W},E(z_i))\geq 1$ for all $i$.
\end{description}
\textbf{Case 1.} By reordering, we may assume that
$cr_\phi({W},E(z_1))=0.$ $\phi(W)$ divides $\mathbb{R}^2$ into
regions and the condition $cr_\phi({W},E(z_1))=0$ implies that
$\phi(X\cup Y\cup S\cup T)$ is contained in the boundary of one of
the regions. Figure 18 shows all the possible drawings of $\phi(W)$.

\begin{center}
\begin{picture}(0,0)(70,10)
\put(0,0){$\bullet$} \put(41,0){$\circ$} \put(40,-40){$\bullet$}
\put(0,-40){$\circ$} \put(-19,-20){$\bullet$}

\put(2,3){\line(1,0){40}} \put(2,3){\line(1,-1){40}}
\put(2,3){\line(0,-1){39}}

\put(42,-37){\line(-1,0){38}} \put(42,-37){\line(0,1){40}}

\put(2,-36){\line(-1,1){18}}

\put(-16,-16){\line(1,1){20}} \put(-16,-17){\line(3,1){58}}
\put(-16,-17){\line(3,-1){60}}

\put(10,-60){$D_1$}
\end{picture}
\begin{picture}(0,0)(-40,10)
\put(0,0){$\bullet$} \put(40,0){$\bullet$} \put(40,-40){$\bullet$}
\put(0,-40){$\circ$} \put(-20,-20){$\circ$}

\put(2,3){\line(1,0){40}} \put(2,3){\line(1,-1){40}}
\put(2,3){\line(0,-1){39}}

\put(42,-37){\line(-1,0){38}} \put(42,-37){\line(0,1){40}}

\put(4,-36){\line(1,1){38}}

\put(-16,-16){\line(1,1){20}} \put(-16,-17){\line(3,1){60}}
\put(-16,-17){\line(3,-1){60}}

\put(10,-60){$D_2$}

\put(-70,-75){Figure 18.}
\end{picture}
\end{center}
$$$$\\\\\\\\\\\\

Denote $F=W\cup E(z_1)$. Then Figure 19 has shown all the possible
drawings of $F$. Therefore, for each drawing $D_i$ ($1\leq i\leq 2$)
of $F$, no matter which regions of $z_j$ ($2\leq j\leq n$) is
located, we have
\begin{eqnarray}\label{Z6}
cr_\phi(F,E(z_j))\geq 4.
\end{eqnarray}

\begin{center}
\begin{picture}(0,0)(70,22)
\put(0,0){$\bullet$} \put(41,0){$\circ$} \put(40,-40){$\bullet$}
\put(0,-40){$\circ$} \put(-19,-20){$\bullet$}

\put(2,3){\line(1,0){40}} \put(2,3){\line(1,-1){40}}
\put(2,3){\line(0,-1){39}}

\put(42,-37){\line(-1,0){38}} \put(42,-37){\line(0,1){40}}

\put(2,-36){\line(-1,1){18}}

\put(-16,-16){\line(1,1){20}} \put(-16,-17){\line(3,1){58}}
\put(-16,-17){\line(3,-1){60}}

\put(-40,-20){$\star$}

\put(-38,-18){\line(1,0){20}} \put(-38,-18){\line(2,1){40}}
\put(-38,-18){\line(2,-1){39}}  \put(-38,-18){\line(1,1){40}}
\put(-38,-18){\line(1,-1){40}} \put(2,22){\line(1,0){40}}
\put(2,-58){\line(1,0){40}} \put(42,22){\line(0,-1){20}}
\put(42,-58){\line(0,1){20}}

\put(10,-70){$D_1$}
\end{picture}
\begin{picture}(0,0)(-40,22)
\put(0,0){$\bullet$} \put(40,0){$\bullet$} \put(40,-40){$\bullet$}
\put(0,-40){$\circ$} \put(-20,-20){$\circ$}

\put(2,3){\line(1,0){40}} \put(2,3){\line(1,-1){40}}
\put(2,3){\line(0,-1){39}}

\put(42,-37){\line(-1,0){38}} \put(42,-37){\line(0,1){40}}

\put(4,-36){\line(1,1){38}}

\put(-16,-16){\line(1,1){20}} \put(-16,-17){\line(3,1){60}}
\put(-16,-17){\line(3,-1){60}}

\put(-40,-20){$\star$}

\put(-38,-18){\line(1,0){19}} \put(-38,-18){\line(2,1){40}}
\put(-38,-18){\line(2,-1){39}}  \put(-38,-18){\line(1,1){40}}
\put(-38,-18){\line(1,-1){40}} \put(2,22){\line(1,0){40}}
\put(2,-58){\line(1,0){40}} \put(42,22){\line(0,-1){20}}
\put(42,-58){\line(0,1){20}}

\put(10,-70){$D_2$}

\put(-70,-90){Figure 19.}
\end{picture}
\end{center}
$$$$\\\\\\\\\\\\\\

Note that (\ref{10}) and (\ref{34}) are true for $\phi$. Since
$cr_\phi(F)\geq 3$ in $D_i$ ($1\leq i\leq 3$),  by (\ref{10}),
(\ref{34}), (\ref{Z6}), we have $cr_\phi(E)\geq
3+Z(5,n-1)+4(n-1)\geq Z(5,n)+2n$ which contradicts to
(\ref{Z3}).\\
\textbf{Case 2.} Since $cr_\phi({W},E(z_i))\leq 1$ for some $i$, we
have $cr_\phi({W},E(z_i))=1$ for some $i$. We may assume that
\begin{eqnarray}\label{Z9}
cr_\phi({W},E(z_1))=1.
\end{eqnarray}
Therefore there must be at least 4 vertices of $\phi(X\cup Y\cup
S\cup T)$ is contained in the boundary of one of the regions in
$\phi(W)$. Note that $K_{1,1,1,2}$ can be obtained by deleting an
edge in $K_{1,1,1,1,1}\cong K_5$. Therefore, the drawings of
$\phi(W)$ can be obtained by deleting an edge in a drawing of $K_5$
such that one of the regions contains at least 3 vertices. One can
check that the only possible drawings of $K_5$ such that one of the
regions contains at least 3 vertices are shown in Figure 20.

\begin{center}
\begin{picture}(0,0)(130,60)
\put(0,0){$\bullet$} \put(40,0){$\bullet$} \put(60,30){$\bullet$}
\put(-20,30){$\bullet$} \put(20,60){$\bullet$}

\put(2,2){\line(1,0){40}} \put(2,2){\line(2,1){60}}
\put(2,2){\line(1,3){20}} \put(2,2){\line(-2,3){20}}

\put(42,2){\line(2,3){20}} \put(42,2){\line(-1,3){20}}
\put(42,2){\line(-2,1){60}}

\put(22,62){\line(4,-3){40}} \put(22,62){\line(-4,-3){40}}

\put(-18,32){\line(1,0){80}}

\end{picture}
\begin{picture}(0,0)(35,60)
\put(0,0){$\bullet$} \put(60,0){$\bullet$} \put(0,60){$\bullet$}
\put(60,60){$\bullet$} \put(30,45){$\bullet$}

\put(2,2){\line(1,0){60}} \put(2,2){\line(0,1){60}}
\put(2,2){\line(1,1){60}}

\put(62,2){\line(0,1){60}} \put(62,2){\line(-1,1){60}}

\put(62,63){\line(-1,0){60}}

\put(32,47){\line(2,1){30}} \put(32,47){\line(-2,1){30}}
\put(33,47){\line(2,-3){30}} \put(32,47){\line(-2,-3){30}}
\end{picture}
\begin{picture}(0,0)(-65,60)
\put(0,0){$\bullet$} \put(80,0){$\bullet$} \put(40,60){$\bullet$}

\put(55,20){$\bullet$} \put(25,20){$\bullet$}

\put(2,2){\line(1,0){80}} \put(2,2){\line(2,3){40}}
\put(83,2){\line(-2,3){40}}

\put(58,22){\line(-2,5){16}} \put(58,22){\line(6,-5){23}}
\put(58,21){\line(-3,-1){57}}

\put(27,22){\line(2,5){16}} \put(27,22){\line(-6,-5){23}}
\put(27,21){\line(3,-1){57}}

\put(27,22){\line(1,0){30}}
\end{picture}
\end{center}
$$$$\\\\
\begin{center}
Figure 20.
\end{center}

To obtain the possible drawing of $\phi(W)$, we have to choose an
edge and delete it such that the resulting the drawing has a region
containing at least 4 vertices. One can check that the only possible
drawings of $\phi(W)$ are shown in Figure 18 and 21.

\begin{center}
\begin{picture}(0,0)(130,60)

\put(0,0){$\bullet$} \put(60,0){$\bullet$} \put(0,60){$\bullet$}
\put(60,60){$\bullet$} \put(30,45){$\bullet$}

 \put(2,2){\line(0,1){60}}
\put(2,2){\line(1,1){60}}

\put(62,2){\line(0,1){60}} \put(62,2){\line(-1,1){60}}

\put(62,63){\line(-1,0){60}}

\put(32,47){\line(2,1){30}} \put(32,47){\line(-2,1){30}}
\put(33,47){\line(2,-3){30}} \put(32,47){\line(-2,-3){30}}

\put(27,-14){$D_3$}
\end{picture}
\begin{picture}(0,0)(35,60)
\put(0,0){$\bullet$} \put(60,0){$\bullet$} \put(0,60){$\bullet$}
\put(60,60){$\bullet$} \put(30,45){$\bullet$}

\put(2,2){\line(1,0){60}} \put(2,2){\line(0,1){60}}

\put(62,2){\line(0,1){60}} \put(62,2){\line(-1,1){60}}

\put(62,63){\line(-1,0){60}}

\put(32,47){\line(2,1){30}} \put(32,47){\line(-2,1){30}}
\put(33,47){\line(2,-3){30}} \put(32,47){\line(-2,-3){30}}

\put(27,-14){$D_4$}
\end{picture}
\begin{picture}(0,0)(-65,60)
\put(0,0){$\bullet$} \put(60,0){$\bullet$} \put(0,60){$\bullet$}
\put(60,60){$\bullet$} \put(30,45){$\bullet$}

\put(2,2){\line(1,0){60}} \put(2,2){\line(0,1){60}}
\put(2,2){\line(1,1){60}}

\put(62,2){\line(0,1){60}} \put(62,2){\line(-1,1){60}}

\put(62,63){\line(-1,0){60}}

\put(32,47){\line(2,1){30}} \put(32,47){\line(-2,1){30}}
\put(32,47){\line(-2,-3){30}}

\put(27,-14){$D_5$}
\end{picture}
\end{center}
$$$$\\\\\\

\begin{center}
\begin{picture}(0,0)(80,60)
\put(0,0){$\bullet$} \put(60,0){$\bullet$} \put(0,60){$\bullet$}
\put(60,60){$\bullet$} \put(30,45){$\bullet$}

\put(2,2){\line(1,0){60}} \put(2,2){\line(0,1){60}}
\put(2,2){\line(1,1){60}}

\put(62,2){\line(0,1){60}} \put(62,2){\line(-1,1){60}}

\put(62,63){\line(-1,0){60}}

 \put(32,47){\line(-2,1){30}}
\put(33,47){\line(2,-3){30}} \put(32,47){\line(-2,-3){30}}

\put(27,-14){$D_6$}
\end{picture}
\begin{picture}(0,0)(-15,60)
\put(0,0){$\bullet$} \put(60,0){$\bullet$} \put(0,60){$\bullet$}
\put(60,60){$\bullet$} \put(30,45){$\bullet$}

\put(2,2){\line(1,0){60}} \put(2,2){\line(1,1){60}}

\put(2,2){\line(0,1){60}} \put(62,2){\line(-1,1){60}}

\put(62,63){\line(-1,0){60}}

\put(32,47){\line(2,1){30}} \put(32,47){\line(-2,1){30}}
\put(33,47){\line(2,-3){30}} \put(32,47){\line(-2,-3){30}}

\put(27,-14){$D_7$}
\end{picture}
\end{center}
$$$$\\\\\\
\begin{center}
Figure 21.
\end{center}

Note that it is impossible for $\phi(W)$ drawn as in Figure 18, that
is, it is impossible for $\phi(W)=D_1, D_2$. If $\phi(W)=D_1$ or
$D_2$, it is impossible for $cr_\phi(W,E(z_1))=1$ for any good
drawing. Therefore $\phi(W)$ must be drawn as in Figure 21. Denote
$F=W\cup E(z_1)$. Then by (\ref{Z9}), $\phi(F)$ must be drawn as in
Figure 22.

If $\phi(F)=F_1,F_5, F_6$ or $F_7$, then we must have
$cr_\phi(F,E(z_j))\geq 4.$ Then following the same proof in Case 1,
we can obtain a contradiction.

Therefore we can assume that $\phi(F)=F_2,F_3$ or $F_4$. Then if
$z_j$ ($2\leq j\leq n$) is located in the region which does not mark
with $\ast$, we have
\begin{eqnarray}\label{Z10}
cr_\phi(F,E(z_j))\geq 4.
\end{eqnarray}
If $z_j$ ($2\leq j\leq n$) is located in the region marked with
$\ast$, we have
\begin{eqnarray}\label{Z16}
cr_\phi(F,E(z_j))\geq 3.
\end{eqnarray}
$$$$

\begin{center}
\begin{picture}(0,0)(130,80)
\put(0,0){$\bullet$} \put(60,0){$\bullet$} \put(0,60){$\bullet$}
\put(60,60){$\bullet$} \put(30,45){$\bullet$}

 \put(2,2){\line(0,1){60}}
\put(2,2){\line(1,1){60}}

\put(62,2){\line(0,1){60}} \put(62,2){\line(-1,1){60}}

\put(62,63){\line(-1,0){60}}

\put(32,47){\line(2,1){30}} \put(32,47){\line(-2,1){30}}
\put(33,47){\line(2,-3){30}} \put(32,47){\line(-2,-3){30}}

\put(30,79){$\star$} \put(30,89){$z_1$}

\put(32,81){\line(0,-1){35}} \put(34,81){\line(3,-2){29}}
\put(31,81){\line(-3,-2){29}}

\put(32,81){\line(1,0){45}} \put(77,81){\line(0,-1){80}}
\put(77,1){\line(-1,0){15}}

\put(32,81){\line(-1,0){45}} \put(-13,81){\line(0,-1){80}}
\put(-13,1){\line(1,0){15}}

\put(27,-14){$F_1$}
\end{picture}
\begin{picture}(0,0)(35,80)
\put(0,0){$\bullet$} \put(60,0){$\bullet$} \put(0,60){$\bullet$}
\put(60,60){$\bullet$} \put(30,45){$\bullet$}

\put(2,2){\line(1,0){60}} \put(2,2){\line(0,1){60}}

\put(62,2){\line(0,1){60}} \put(62,2){\line(-1,1){60}}

\put(62,63){\line(-1,0){60}}

\put(32,47){\line(2,1){30}} \put(32,47){\line(-2,1){30}}
\put(33,47){\line(2,-3){30}} \put(32,47){\line(-2,-3){30}}

\put(30,79){$\star$} \put(30,89){$z_1$}

\put(32,81){\line(0,-1){35}} \put(34,81){\line(3,-2){29}}
\put(31,81){\line(-3,-2){29}}

\put(32,81){\line(1,0){45}} \put(77,81){\line(0,-1){80}}
\put(77,1){\line(-1,0){15}}

\put(32,81){\line(-1,0){45}} \put(-13,81){\line(0,-1){80}}
\put(-13,1){\line(1,0){15}}

\put(27,-14){$F_2$}

\put(30,35){$\ast$}

\end{picture}
\begin{picture}(0,0)(-65,80)
\put(0,0){$\bullet$} \put(60,0){$\bullet$} \put(0,60){$\bullet$}
\put(60,60){$\bullet$} \put(30,45){$\bullet$}

\put(2,2){\line(1,0){60}} \put(2,2){\line(0,1){60}}

\put(62,2){\line(0,1){60}} \put(62,2){\line(-1,1){60}}

\put(62,63){\line(-1,0){60}}

\put(32,47){\line(2,1){30}} \put(32,47){\line(-2,1){30}}
\put(33,47){\line(2,-3){30}} \put(32,47){\line(-2,-3){30}}

\put(30,79){$\star$} \put(30,89){$z_1$}

\put(34,81){\line(4,-1){40}} \put(74,71){\line(0,-1){23}}
\put(74,48){\line(-1,0){40}}

\put(34,81){\line(3,-2){29}} \put(31,81){\line(-3,-2){29}}

\put(32,81){\line(1,0){45}} \put(77,81){\line(0,-1){80}}
\put(77,1){\line(-1,0){15}}

\put(32,81){\line(-1,0){45}} \put(-13,81){\line(0,-1){80}}
\put(-13,1){\line(1,0){15}}

\put(27,-14){$F_3$}

\put(22,43){$\ast$}
\end{picture}
\end{center}
$$$$\\\\\\\\\\

\begin{center}
\begin{picture}(0,0)(130,80)
\put(0,0){$\bullet$} \put(60,0){$\bullet$} \put(0,60){$\bullet$}
\put(60,60){$\bullet$} \put(30,45){$\bullet$}

\put(2,2){\line(1,0){60}} \put(2,2){\line(0,1){60}}
\put(2,2){\line(1,1){60}}

\put(62,2){\line(0,1){60}} \put(62,2){\line(-1,1){60}}

\put(62,63){\line(-1,0){60}}

\put(32,47){\line(2,1){30}} \put(32,47){\line(-2,1){30}}
\put(32,47){\line(-2,-3){30}}

\put(30,79){$\star$} \put(30,89){$z_1$}

\put(32,81){\line(0,-1){35}} \put(34,81){\line(3,-2){29}}
\put(31,81){\line(-3,-2){29}}

\put(32,81){\line(1,0){45}} \put(77,81){\line(0,-1){80}}
\put(77,1){\line(-1,0){15}}

\put(32,81){\line(-1,0){45}} \put(-13,81){\line(0,-1){80}}
\put(-13,1){\line(1,0){15}}

\put(27,-14){$F_4$}

\put(30,35){$\ast$}
\end{picture}
\begin{picture}(0,0)(35,80)
\put(30,79){$\star$} \put(30,89){$z_1$}

\put(32,81){\line(0,-1){35}} \put(34,81){\line(3,-2){29}}
\put(31,81){\line(-3,-2){29}}

\put(32,81){\line(1,0){45}} \put(77,81){\line(0,-1){80}}
\put(77,1){\line(-1,0){15}}

\put(32,81){\line(-1,0){45}} \put(-13,81){\line(0,-1){80}}
\put(-13,1){\line(1,0){15}}

\put(0,0){$\bullet$} \put(60,0){$\bullet$} \put(0,60){$\bullet$}
\put(60,60){$\bullet$} \put(30,45){$\bullet$}

\put(2,2){\line(1,0){60}} \put(2,2){\line(0,1){60}}
\put(2,2){\line(1,1){60}}

\put(62,2){\line(0,1){60}} \put(62,2){\line(-1,1){60}}

\put(62,63){\line(-1,0){60}}

 \put(32,47){\line(-2,1){30}}
\put(33,47){\line(2,-3){30}} \put(32,47){\line(-2,-3){30}}

\put(27,-14){$F_5$}
\end{picture}
\begin{picture}(0,0)(-65,80)
\put(30,79){$\star$} \put(30,89){$z_1$}

\put(32,81){\line(0,-1){35}} \put(34,81){\line(3,-2){29}}
\put(31,81){\line(-3,-2){29}}

\put(32,81){\line(1,0){45}} \put(77,81){\line(0,-1){80}}
\put(77,1){\line(-1,0){15}}

\put(32,81){\line(-1,0){45}} \put(-13,81){\line(0,-1){80}}
\put(-13,1){\line(1,0){15}}

\put(0,0){$\bullet$} \put(60,0){$\bullet$} \put(0,60){$\bullet$}
\put(60,60){$\bullet$} \put(30,45){$\bullet$}

\put(2,2){\line(1,0){60}} \put(2,2){\line(1,1){60}}

\put(2,2){\line(0,1){60}} \put(62,2){\line(-1,1){60}}

\put(62,63){\line(-1,0){60}}

\put(32,47){\line(2,1){30}} \put(32,47){\line(-2,1){30}}
\put(33,47){\line(2,-3){30}} \put(32,47){\line(-2,-3){30}}

\put(27,-14){$F_6$}
\end{picture}
\end{center}
$$$$\\\\\\\\\\

\begin{center}
\begin{picture}(0,0)(35,80)
\put(30,79){$\star$} \put(30,89){$z_1$}

\put(34,81){\line(4,-1){40}} \put(74,71){\line(0,-1){23}}
\put(74,48){\line(-1,0){40}}

\put(34,81){\line(3,-2){29}} \put(31,81){\line(-3,-2){29}}

\put(32,81){\line(1,0){45}} \put(77,81){\line(0,-1){80}}
\put(77,1){\line(-1,0){15}}

\put(32,81){\line(-1,0){45}} \put(-13,81){\line(0,-1){80}}
\put(-13,1){\line(1,0){15}}

\put(0,0){$\bullet$} \put(60,0){$\bullet$} \put(0,60){$\bullet$}
\put(60,60){$\bullet$} \put(30,45){$\bullet$}

\put(2,2){\line(1,0){60}} \put(2,2){\line(1,1){60}}

\put(2,2){\line(0,1){60}} \put(62,2){\line(-1,1){60}}

\put(62,63){\line(-1,0){60}}

\put(32,47){\line(2,1){30}} \put(32,47){\line(-2,1){30}}
\put(33,47){\line(2,-3){30}} \put(32,47){\line(-2,-3){30}}

\put(27,-14){$F_7$}
\end{picture}
\end{center}
$$$$\\\\\\\\\\
\begin{center}
Figure 22.
\end{center}

We claim that it is impossible for $cr_\phi(F,E(z_j))=3$. Suppose
not, we may assume that $z_2$ is located in the region marked with
$\ast$ such that $cr_\phi(F,E(z_2))=3$. Then we must have
\begin{eqnarray}\label{Z11}
cr_\phi(W,E(z_2))=3\mbox{ and }cr_\phi(E(z_1),E(z_2))=0.
\end{eqnarray}
For $3\leq k\leq n$, $\langle E(z_1)\cup E(z_2)\cup
E(z_k)\rangle\cong K_{3,5}$. Hence, from (\ref{21}), (\ref{3}),
(\ref{4}), we have
\begin{eqnarray}\label{Z12}
cr_\phi(E(z_1)\cup E(z_2),E(z_k))+cr_\phi(E(z_1),E(z_2))\geq 4.
\end{eqnarray}
Then by (\ref{Z11}) and (\ref{Z12}), we have
\begin{eqnarray}\label{Z15}
cr_\phi(E(z_1)\cup E(z_2),E(z_k))\geq 4&\mbox{for}&3\leq k\leq n.
\end{eqnarray}

Let $E'=E-(E(z_1)\cup E(z_2))$. Then by (\ref{3}), (\ref{4}),
\begin{eqnarray}\label{Z14}
\begin{array}{rcl}
cr_\phi(E)&=&cr_\phi(E')+cr_\phi(E(z_1)\cup
E(z_2))+cr_\phi(W,E(z_1))\\
&&+cr_\phi(W,E(z_2))+\displaystyle\sum_{k=3}^n cr_\phi(E(z_1)\cup
E(z_2),E(z_k)).
\end{array}
\end{eqnarray}

Then by the fact that $\langle E'\rangle\cong K_{1,1,1,2,n-2}$ and
induction assumption, (\ref{Z9}), (\ref{Z11}), (\ref{Z15}),
(\ref{Z14}), we have $cr_\phi(E)\geq Z(5,n-2)+2(n-2)+1+3+4(n-2)\geq
Z(5,n)+2n$ which contradicts (\ref{Z3}). This proves our claim.

By the claim and (\ref{Z16}), we know that if $z_j$ ($2\leq j\leq
n$) is located in the region marked with $\ast$, we have
\begin{eqnarray}\label{Z17}
cr_\phi(F,E(z_j))\geq 4.
\end{eqnarray}
Then by (\ref{Z10}) and (\ref{Z17}), we know that no matter which
region $z_j$ is located, we have $cr_\phi(F,E(z_j))\geq 4.$ Then
following the same proof in Case 1, we can obtain a contradiction.
\end{proof}

\section{Crossing number of $K_{1,4,n}$}

In this section, we will prove
\begin{thm}
The crossing number of the complete 3-partite graph $K_{1,4,n}$ is
given by  $$cr(K_{1,4,n})=Z(5,n)+2\lfloor\frac{n}{2}\rfloor.$$
\end{thm}
\begin{proof}
Let $(X,Y,Z)$ be partition of $K_{1,4,n}$ such that $X=\{ x_1\}$,
$Y=\displaystyle\bigcup_{i=1}^4\{y_i\}$ and
$Z=\displaystyle\bigcup_{i=1}^n\{z_i\}$. To show that
$cr(K_{1,4,n})\leq Z(5,n)+2\displaystyle\lfloor\frac{n}{2}\rfloor,$
see Figure 23 for $n=4$, and it can be easily generalized to $n$.

\begin{center}
\begin{picture}(60,20)(-25,35)
\put(0,0){$\bullet$} \put(1,40){$\bullet$} \put(-11,-40){$\bullet$}
\put(0,-81){$\bullet$} \put(0,-121){$\bullet$}

\put(-9,8){$x_1$} \put(-13,41){$y_1$} \put(-17,-30){$y_2$}
\put(-12,-81){$y_3$} \put(10,-121){$y_4$}

\put(2,2){\line(0,-1){80}} \put(2,2){\line(-1,-4){10}}
 \put(3,2){\line(0,1){40}}

\put(20,-20){$\bullet$} \put(60,-22){$\bullet$}
\put(-40,-20){$\bullet$} \put(-80,-21){$\bullet$}

\put(8,-18){$z_1$} \put(-49,-20){$z_3$} \put(67,-20){$z_2$}
\put(-89,-21){$z_4$}

\put(22,-18){\line(-1,1){20}} \put(24,-18){\line(-1,3){20}}
\put(24,-17){\line(-3,-2){32}} \put(24,-18){\line(-1,-3){20}}

\put(62,-18){\line(-3,1){60}} \put(64,-18){\line(-1,1){60}}
 \put(64,-20){\line(-4,-1){71}} \put(64,-18){\line(-1,-1){60}}

\put(-38,-18){\line(2,1){40}} \put(-38,-18){\line(2,3){40}}
\put(-38,-18){\line(3,-2){28}} \put(-38,-18){\line(2,-3){40}}

\put(-78,-18){\line(4,1){80}} \put(-78,-17){\line(4,3){80}}
\put(-78,-20){\line(4,-1){70}} \put(-78,-20){\line(4,-3){80}}

\put(4,-119){\line(1,5){20}} \put(4,-119){\line(3,5){60}}
\put(2,-119){\line(-2,5){40}} \put(2,-119){\line(-2,5){40}}
\put(1,-119){\line(-4,5){80}}

\put(126,-17){\line(-6,1){123}} \put(126,-17){\line(-6,-5){123}}

 \put(-15,-148){Figure 23.}
\end{picture}
\end{center}
$$$$\\\\\\\\\\\\\\\\\\\\\\\\\\

Therefore it suffices to show that
\begin{eqnarray}\label{A2}
cr(K_{1,4,n})\geq Z(5,n)+2\lfloor\frac{n}{2}\rfloor.
\end{eqnarray}

We will prove (\ref{A2}) by induction on $n$. It is clear that
(\ref{A2}) is true for $n=1$. For $n=2$, since $K_{1,4,2}$ contains
$K_{3,4}$, $cr(K_{1,4,2})\geq cr(K_{3,4})=2$ by (\ref{21}).
Therefore, (\ref{A2}) is true for $n=2$. Now we can assume that
$n\geq 3$. Suppose (\ref{A2}) is true for all value less than $n$
and is not true for $n$. Thus there exists a good immersion $\phi$
of $K_{1,4,n}$ such that
\begin{eqnarray}\label{A6}
cr_{\phi}(E)<Z(5,n)+2\lfloor\frac{n}{2}\rfloor.
\end{eqnarray}

Note that the drawing of $K_{1,4,n}$, $E$, includes $n$ drawings
of $K_{1,4,n-1}$, each obtained by suppressing one vertex $z_i$ in
$Z$. Each crossing between $u_iz_j$ and $u_kz_l$ where $u_i,
u_k\in X\cup Y$ and $z_j, z_l\in Z$ will occur in $n-2$ of these
drawings, namely, those in which neither $z_i$ nor $z_l$ is
suppressed. Also, each crossing between $u_iu_j$ and $u_kz_l$
where $u_i, u_j, u_k\in X\cup Y$ and $z_l\in Z$ will occur in
$n-1$ of these drawings, namely, those in which $z_l$ is not
suppressed. Therefore, by the fact that $cr_\phi(E_{XY})=0$, it
follows that
\begin{eqnarray}\label{A37}
(n-1)cr_\phi(E_{XY},\bigcup_{i=1}^nE(z_i))+(n-2)cr_\phi(\bigcup_{i=1}^nE(z_i))\geq
n\mbox{ }cr(K_{1,4,n-1})
\end{eqnarray}
Then by induction assumption and (\ref{A37}), we have
\begin{eqnarray}\label{A38}
cr_\phi(E)\geq
Z(5,n-1)+2\lfloor\frac{n-1}{2}\rfloor-\sum_{i=1}^n\frac{cr_\phi(E_{XY},E(z_i))}{n-2}.
\end{eqnarray}
By (\ref{3}) and (\ref{4}), we have
\begin{eqnarray}\label{A5}
cr_\phi(E)=cr_\phi({E_{XY}})+cr_\phi(\bigcup_{i=1}^n
E(z_i))+\sum_{i=1}^n cr_\phi({E_{XY}},E(z_i)).
\end{eqnarray}
Since $\langle\displaystyle\bigcup_{i=1}^n E(z_i)\rangle\cong
K_{5,n}$, by (\ref{21}), we have
\begin{eqnarray}\label{A7}
cr_\phi(\bigcup_{i=1}^n E(z_i))\geq Z(5,n).
\end{eqnarray}
Therefore, by (\ref{A6}), (\ref{A5}) and (\ref{A7}), we have
\begin{eqnarray}\label{A39}
\sum_{i=1}^ncr_\phi(E_{XY},E(z_i))\leq
2\lfloor\frac{n}{2}\rfloor-1.
\end{eqnarray}
Combining (\ref{A38}), (\ref{A39}) and the fact that $cr_\phi(E)$
is an integer, we have
\begin{eqnarray}\label{A40}
cr_\phi(E)\geq Z(5,n)+2\lfloor\frac{n}{2}\rfloor-1.
\end{eqnarray}
By (\ref{A6}), (\ref{A40}) and the fact that $cr_\phi(E)$ is an
integer, we have
\begin{eqnarray}\label{A41}
cr_\phi(E)=cr_\phi(\bigcup_{i=1}^n
E(z_i))+\sum_{i=1}^ncr_\phi(E_{XY},E(z_i))=Z(5,n)+2\lfloor\frac{n}{2}\rfloor-1.
\end{eqnarray}

We may assume that $\phi(\langle E_{XY}\rangle)$ is drawn as in
Figure 24.

\begin{center}
\begin{picture}(0,0)(120,5)
\put(0,-20){$\circ$} \put(9,-20){$x_1$}

\put(4,-16){\line(1,1){19}} \put(4,-18){\line(1,-1){19}}
\put(1,-16){\line(-1,1){19}} \put(1,-18){\line(-1,-1){19}}

\put(20,0){$\bullet$} \put(28,0){$y_2$}

\put(-20,0){$\bullet$} \put(-30,0){$y_1$}

\put(20,-40){$\bullet$} \put(28,-40){$y_3$}

\put(-20,-40){$\bullet$} \put(-30,-40){$y_4$}

\put(-20,-70){Figure 24.}
\end{picture}
\begin{picture}(0,0)(0,5)
\put(0,-20){$\circ$} \put(9,-20){$x_1$}

\put(4,-16){\line(1,1){19}} \put(4,-18){\line(1,-1){19}}
\put(1,-16){\line(-1,1){19}} \put(1,-18){\line(-1,-1){19}}

\put(20,0){$\bullet$} \put(28,0){$y_2$}

\put(-20,0){$\bullet$} \put(-30,0){$y_1$}

\put(20,-40){$\bullet$} \put(28,-40){$y_3$}

\put(-20,-40){$\bullet$} \put(-30,-40){$y_4$}

\put(-40,-20){$\star$} \put(-50,-20){$z_i$}
\put(-38,-17){\line(1,0){39}}

\put(-20,-67){Figure 25.}
\end{picture}
\begin{picture}(0,0)(-120,5)
\put(0,-20){$\circ$} \put(9,-20){$x_1$}

\put(4,-16){\line(1,1){19}} \put(4,-18){\line(1,-1){19}}
\put(1,-16){\line(-1,1){19}} \put(1,-18){\line(-1,-1){19}}

\put(20,0){$\bullet$} \put(28,0){$y_2$}

\put(-20,0){$\bullet$} \put(-30,0){$y_1$}

\put(20,-40){$\bullet$} \put(28,-40){$y_3$}

\put(-20,-40){$\bullet$} \put(-30,-40){$y_4$}

\put(0,-11){$\star$} \put(-6,-3){$z_{n+1}$}

\put(-20,-65){Figure 26.}
\end{picture}
\end{center}
$$$$\\\\\\\\

For $1\leq j\leq 4$, let $A_j$ be the set of $z_i$, $1\leq i\leq n$,
such that $x_1z_i$ lies between the edges $y_j$ and $y_{j+1}$ (mod 4
for $j+1$). See Figure 25 for $z_i\in A_4$.

For $n$ is even, we may assume that $|A_1|\leq |A_3|$. For $n$ is
odd, we have $|A_1|\neq |A_3|$ or $|A_2|\neq |A_4|$ (Otherwise,
$n$ is even). Then we may assume $|A_1|< |A_3|$. Therefore, we
have
\begin{eqnarray}\label{A43}
|A_1|\leq |A_3|-\lceil\frac{n}{2}\rceil+\lfloor\frac{n}{2}\rfloor
\end{eqnarray}

We are going to obtain a drawing of $K_{5,n+1}$ from $\phi(E)$ and
then obtain a contradiction. To obtain a drawing of $K_{5,n+1}$ from
$\phi(E)$, we draw a new vertex, denoted it by $z_{n+1}$, near the
vertex $x_1$ and lying in the region between the edges $x_1y_1$ and
$x_1y_2$, as shown in Figure 26.\\

\begin{center}
\begin{picture}(0,0)(140,0)
\put(0,-20){$\circ$} \put(0,-27){$x_1$}

\put(20,0){$\bullet$} \put(28,0){$y_2$}

\put(-20,0){$\bullet$} \put(-30,0){$y_1$}

\put(20,-40){$\bullet$} \put(28,-40){$y_3$}

\put(-20,-40){$\bullet$} \put(-30,-40){$y_4$}

\put(0,-11){$\star$} \put(-6,-0){$z_{n+1}$}

\put(2,-9){\line(2,1){20}} \put(2,-9){\line(-2,1){20}}
\put(2,-9){\line(2,-3){19}} \put(2,-9){\line(-2,-3){19}}

\put(2,-9){\line(0,-1){7}}

\put(-20,-70){Figure 27.}
\end{picture}
\begin{picture}(0,0)(53,0)
\put(0,-20){$\circ$} \put(0,-27){$x_1$}

\put(20,0){$\bullet$} \put(28,0){$y_2$}

\put(-20,0){$\bullet$} \put(-30,0){$y_1$}

\put(20,-40){$\bullet$} \put(28,-40){$y_3$}

\put(-20,-40){$\bullet$} \put(-30,-40){$y_4$}

\put(0,-11){$\star$} \put(7,-13){$z_{n+1}$}

\put(2,-9){\line(2,1){20}} \put(2,-9){\line(-2,1){20}}
\put(2,-9){\line(2,-3){19}} \put(2,-9){\line(-2,-3){19}}

\put(2,-9){\line(0,-1){7}}

\put(-10,0){$\star$} \put(-3,3){$z_i$}

\put(-7,2){\line(1,-2){9}}

\put(-20,-70){Figure 28.}
\end{picture}
\begin{picture}(0,0)(-40,0)
\put(0,-20){$\circ$} \put(0,-27){$x_1$}

\put(20,0){$\bullet$} \put(28,0){$y_2$}

\put(-20,0){$\bullet$} \put(-30,0){$y_1$}

\put(20,-40){$\bullet$} \put(28,-40){$y_3$}

\put(-20,-40){$\bullet$} \put(-30,-40){$y_4$}

\put(0,-11){$\star$} \put(-6,-0){$z_{n+1}$}

\put(2,-9){\line(2,1){20}} \put(2,-9){\line(-2,1){20}}
\put(2,-9){\line(2,-3){19}} \put(2,-9){\line(-2,-3){19}}

\put(2,-9){\line(0,-1){7}}

\put(-40,-20){$\star$} \put(-50,-20){$z_i$}
\put(-38,-17){\line(1,0){39}}

\put(-20,-70){Figure 29.}
\end{picture}
\begin{picture}(0,0)(-130,0)
\put(0,-20){$\circ$} \put(1,-25){$x_1$}

\put(20,0){$\bullet$} \put(28,0){$y_2$}

\put(-20,0){$\bullet$} \put(-30,0){$y_1$}

\put(20,-40){$\bullet$} \put(28,-40){$y_3$}

\put(-20,-40){$\bullet$} \put(-30,-40){$y_4$}

\put(0,-11){$\star$} \put(-6,-0){$z_{n+1}$}

\put(2,-9){\line(2,1){20}} \put(2,-9){\line(-2,1){20}}
\put(2,-9){\line(2,-3){19}} \put(2,-9){\line(-2,-3){19}}

\put(2,-9){\line(0,-1){7}}

\put(-10,-45){$\star$} \put(-2,-47){$z_i$}
\put(-7,-42){\line(1,3){8}}

\put(-20,-70){Figure 30.}
\end{picture}
\end{center}
$$$$\\\\\\\\

For $1\leq j\leq 4$, draw the edge $z_{n+1}y_j$ next to the edge
$x_{1}y_j$ and draw the edge $z_{n+1}x_1$ without crossing any edges
in $E$. Then remove the edges $x_1y_j$ from $E$ where $1\leq j\leq
4$. See Figure 27.

Now we have a drawing of $K_{5,n+1}$ with $\{
x_1,y_1,y_2,y_3,y_4\}$ as the partition with 5 vertices and
$\displaystyle\bigcup_{i=1}^{n+1}\{ z_i\}$ as the partition with
$n+1$ vertices, denote the immersion of $K_{5,n+1}$ by $\phi'$.

Note that if $z_i\in A_1$ (see Figure 28), then
\begin{eqnarray}\label{A45}
cr_{\phi'}(E(z_i),E(z_{n+1}))=cr_\phi(E(z_i),E_{XY})+2.
\end{eqnarray}
If $z_i\in A_2\cup A_4$ (see Figure 29), then
\begin{eqnarray}\label{A46}
cr_{\phi'}(E(z_i),E(z_{n+1}))=cr_\phi(E(z_i),E_{XY})+1.
\end{eqnarray}
If $z_i\in A_3$ (see Figure 30), then
\begin{eqnarray}\label{A47}
cr_{\phi'}(E(z_i),E(z_{n+1}))=cr_\phi(E(z_i),E_{XY}).
\end{eqnarray}
Note also that
\begin{eqnarray}\label{A48}
cr_{\phi'}(\displaystyle\bigcup_{i=1}^n E(z_i))
=cr_\phi(\bigcup_{i=1}^n E(z_i))
\end{eqnarray}

Then by (\ref{3}), (\ref{4}), (\ref{A45}), (\ref{A46}),
(\ref{A47}) and (\ref{A48}), we have the crossing number of
$\phi'$ is
\begin{eqnarray*}
&&cr_\phi(\displaystyle\bigcup_{i=1}^n
E(z_i))+\sum_{i=1}^ncr_\phi(E_{XY},E(z_i))+2|A_1|+|A_2|+|A_4|\\
&\leq&
Z(5,n)+2\displaystyle\lfloor\frac{n}{2}\rfloor-1+|A_1|+|A_2|+|A_3|+|A_4|-\lceil
\frac{n}{2}\rceil+\lfloor\frac{n}{2}\rfloor\\
&=&Z(5,n)+2\displaystyle\lfloor\frac{n}{2}\rfloor-1+n-\lceil
\frac{n}{2}\rceil+\lfloor\frac{n}{2}\rfloor,
\end{eqnarray*}
where the second inequality follows from (\ref{A41}), (\ref{A43}),
and the last equality follows from the fact that
$|A_1|+|A_2|+|A_3|+|A_4|=n$. However, since
$Z(5,n)+2\displaystyle\lfloor\frac{n}{2}\rfloor-1+n-\lceil
\frac{n}{2}\rceil+\lfloor\frac{n}{2}\rfloor<Z(5,n+1)$, we obtained
a drawing of $K_{5,n+1}$ with crossing number less than
$Z(5,n+1)$, which is contradicted to (\ref{21}).
\end{proof}

\section{Crossing number of $K_{1,3,n}$}
In this section, we use the basic counting argument to determine the
crossing number of $K_{1,3,n}$ which is obtained by Asano in
cite{Asano}.
\begin{thm}
The crossing number of the complete 3-partite graph $K_{1,3,n}$ is
given by  $$cr(K_{1,3,n})=Z(4,n)+\lfloor\frac{n}{2}\rfloor.$$
\end{thm}
\begin{proof}
As proved in \cite{Asano}, one has $cr(K_{1,3,n})\leq
Z(4,n)+\displaystyle\lfloor\frac{n}{2}\rfloor.$ (Actually, to show
this, one can try to draw $K_{1,3,n}$ with such crossing numbers in
a similar way in the previous sections.) Therefore it suffices to
show that
\begin{eqnarray}\label{C1}
 cr(K_{1,3,n})\geq Z(4,n)+\displaystyle\lfloor\frac{n}{2}\rfloor.
\end{eqnarray}

We will prove (\ref{C1}) by induction on $n$. It is clear that
(\ref{C1}) is true for $n=1$. For $n=2$, note that $K_{1,3,2}$
contains $K_{3,3}$, which gives $cr(K_{1,3,2})\geq cr(K_{3,3})=1$.
Therefore (\ref{C1}) is true for $n=2$. Now we can assume that
$n\geq 3$. Suppose (\ref{C1}) is true for all value less than $n$
and is not true for $n$. Then there exists a good immersion $\phi$
of $K_{1,3,n}$ such that
\begin{eqnarray}\label{C3}
cr_{\phi}(E)<Z(4,n)+\lfloor\frac{n}{2}\rfloor.
\end{eqnarray}
Note that (\ref{5}) and (\ref{7}) are also true for $\phi$.
Therefore, (\ref{5}), (\ref{7}), (\ref{C3}) and the fact that
$cr_\phi(E_{XY})=0$ gives
\begin{eqnarray}\label{C6}
\sum_{i=1}^n cr_\phi({E_{XY}},E(z_i))\leq\lfloor\frac{n}{2}\rfloor
-1.
\end{eqnarray}

Now, we will apply the counting argument. Note that the drawing of
$K_{1,3,n}$, $E$, includes $n$ drawings of $K_{1,3,n-1}$, each
obtained by suppressing one vertex $z_i$ in $Z$. Each crossing
between $u_iz_j$ and $u_kz_l$ where $u_i, u_k\in X\cup Y$ and $z_j,
z_l\in Z$ will occur in $n-2$ of these drawings, namely, those in
which neither $z_j$ nor $z_l$ is suppressed. Also, each crossing
between $u_iu_j$ and $u_kz_l$ where $u_i, u_j, u_k\in X\cup Y$ and
$z_l\in Z$ will occur in $n-1$ of these drawings, namely, those in
which $z_l$ is not suppressed. Therefore, by the fact that
$cr_\phi(E_{XY})=0$, it follows that
\begin{eqnarray}\label{C7}
(n-1)cr_\phi(E_{XY},\bigcup_{i=1}^nE(z_i))+(n-2)cr_\phi(\bigcup_{i=1}^nE(z_i))\geq
n\mbox{ }cr(K_{1,3,n-1}).
\end{eqnarray}
Therefore, by induction assumption, (\ref{5}), (\ref{C7}) and the
face that $cr_\phi(E_{XY})=0$, we have
\begin{eqnarray}\label{C8}
cr_\phi(E)\geq
Z(4,n)+\lfloor\frac{n}{2}\rfloor-\sum_{i=1}^n\frac{cr_\phi(E_{XY},E(z_i))}{n-2}.
\end{eqnarray}
Since $n\geq 3$ and $cr_\phi(E)$ is an integer, combining (\ref{C6})
and (\ref{C8}), we have $cr_\phi(E)\geq
Z(4,n)+\displaystyle\lfloor\frac{n}{2}\rfloor,$ which contradicts
(\ref{C3}).
\end{proof}

\section{Conclusion}

We conclude this paper by stating the following conjectures:
\begin{conj}\label{conj}
\begin{equation*}
\begin{split}
cr(K_{1,1,3,n})&=Z(5,n)+\displaystyle\lfloor\frac{3n}{2}\rfloor;\\
cr(K_{2,4,n})&=Z(6,n)+2n.
\end{split}
\end{equation*}
\end{conj}
One can easily check that the conjectural values are the upper
bounds. By using similar methods in this paper, we can proved that
the conjecture is true for many drawings of them.

\textit{Note added in the proof:} Conjecture \ref{conj} was solved
in \cite{Ho1} and \cite{Ho2} respectively.


\begin{thebibliography}{19}

\bibitem{Asano} K. Asano,  The crossing number of
$K_{1,3,n}$ and $K_{2,3,n}$, \textit{J. Graph Theory} \textbf{10} (1986), 1-8.

\bibitem{Richter} E. deKlerk, J. Maharry, D. Pasechnik, R. B. Richter, and G.
Salazar, Improved bounds for the crossing numbers of $K_{m,n}$ and
$K_n$. \textit{SIAM J. Discrete Math.} \textbf{20} (2006), 189-202.


\bibitem{Guy} R. Guy,  The decline and fall of Zarankiewicz's theorem,
in Proof Techniques in Graph Theory (F. Harary Ed.), Academic
Press, New York (1969) 63-69.

\bibitem{Ho1} P. T. Ho, The crossing number of $K_{1,1,3,n}$, \textit{Ars Combinatoria.} \textbf{99} (2011), 461-471.

\bibitem{Ho2} P. T. Ho, The crossing number of $K_{2,4,n}$, \textit{Ars Combinatoria.} \textbf{109} (2013), 527-537.


\bibitem{Harborth} H. Harborth,  Parity of numbers of crossings for
complete $n$-partite graphs. \textit{Math. Slovaca} \textbf{26} (1976), 77-95.

\bibitem{Kleitman} D. J. Kleitman,
The crossing number of $K_{5,n}$. \textit{J. Combin. Theory} \textbf{9} (1970),
315-323.

\bibitem{Szekely} L. A. Sz\'{e}kely, A successful concept for measuring non-planarity
of graphs: the crossing number. {\it Discrete Math.} \textbf{276}
(2004), 331-352.

\bibitem{Woodall} D. R. Woodall, Cyclic-order graphs and
Zarankiewicz's crossing number conjecture. \textit{J. Graph Theory} \textbf{17}
(1993), 657-671.
\end{thebibliography}
\end{document}